\documentclass[10pt]{article}
\usepackage{times}
\usepackage{graphicx}
\usepackage[T1]{fontenc}
\usepackage{amsmath,amssymb,amsfonts,amsthm,amscd}
\usepackage{eucal}
\usepackage[english]{babel}
\usepackage[utf8]{inputenc}
\usepackage{enumerate}
\usepackage{hyperref}
\usepackage{url}
\usepackage{mathpazo}
\usepackage{mathrsfs}
\usepackage{stmaryrd}
\usepackage{mathtools}
\usepackage{xcolor}
\usepackage{marvosym}
\usepackage{subcaption}
\usepackage[margin=1.0in]{geometry}
\theoremstyle{plain}
\numberwithin{equation}{section}

\newtheorem{theorem}{Theorem}[section]
\newtheorem{lemma}[theorem]{Lemma}
\newtheorem{corollary}[theorem]{Corollary}
\newtheorem{proposition}[theorem]{Proposition}
\newtheorem{definition}[theorem]{Definition}
\newtheorem{remark}[theorem]{Remark}

\newtheorem{Notation}[theorem]{Notation}

\newcommand{\beq}{\begin{equation}}
\newcommand{\eeq}{\end{equation}}

\newcommand{\tn}[1]{\left.{#1}\right|_{t=0}}
\newcommand{\ttau}[1]{\left.{#1}\right|_{t=\tau}}
\newcommand{\R}{\mathbb{R}}

\newcommand{\Real}{\mathrm{Re}\,}
\newcommand{\Imag}{\mathrm{Im}\,}

\newcommand{\Id}{\mathrm{Id}}
\newcommand\Sym{\text{Sym}^2(M)}
\newcommand{\g}[2]{\left\langle #1,#2\right\rangle}
\newcommand{\norm}[1]{\left\Vert #1\right\Vert}
\newcommand{\abs}[1]{\left| #1\right|}

\newcommand{\dnu}{\mathrel{\nu\!\nu}}
\newcommand{\p}[1]{\left( #1\right)}
\newcommand{\dt}{\partial_t}
\newcommand{\ddt}{\left.\partial_t\right|_{t=0}\:}

\newcommand\SF{\text{\Gemini}}
\newcommand\MC{\mathcal{H}}

\newcommand\bnab{\overline{\nabla}}
\newcommand\PS{\partial M}
\newcommand\Lie{\mathcal{L}}
\newcommand{\sq}[1]{\left[#1\right]}
\newcommand{\tb}[1]{\left\{#1\right\}}
\newcommand{\pp}[1]{D_{#1}}
\newcommand{\Done}{D_g}
\newcommand{\bal}{\begin{align*} }
\newcommand{\lab}{\end{align*}}

\DeclareMathOperator{\Ric}{Ric}
\DeclareMathOperator{\scal}{scal}

\DeclareMathOperator{\trace}{tr}

\begin{document}
\author{Rasmus Jouttij\"arvi\\ KTH}
\title{Novel Boundary Conditions for the Ricci Flow}
\maketitle

\textbf{Abstract:} If we want to deform a compact Riemannian manifold with boundary using Ricci flow, we first need to decide on appropriate boundary conditions. We would like these conditions to reflect the geometric nature of the flow and allow for a variety of initial data. Importantly, the conditions should be compatible with the expected evolution of Einstein metrics. We propose it's natural to choose those conditions, for which the first variation of certain functionals, such as the Einstein-Hilbert action and Perelman's $\lambda$-functional, does not admit a boundary term. We provide a proof of the short term existence of solutions of the initial boundary value problem, under these conditions.\\

\tableofcontents

\pagebreak

\section{Introduction}

When considering the Ricci flow initial value problem 
\begin{equation}\label{RFIVP}\begin{cases}\partial_t g_t =-2\Ric_{g_t}\\
g_0=g	
\end{cases}\end{equation}
the most natural first step in extending it to a boundary value problem, would be to impose a geometric Dirichlet-type condition on the variation of the induced boundary metric:
\begin{equation}\label{wrongdirichlet}\left.\p{\partial_t g_t}^T\right|_{\PS}=0\end{equation}
where the superscript $T$ here denotes the tangential components of the tensor, while - when applied to a metric - $g^T$ shall mean the induced metric on the boundary. However, this boundary condition poses a problem: In the study of Ricci flow on manifolds without boundary, a certain class of metrics play an important role: Einstein Metrics. That is, metrics satisfying the Einstein equation
\begin{equation}\label{Einsteineq}\Ric_g=\mu g\qquad \qquad \mu\in \R\end{equation}
The importance of these metrics comes from the trivial observation that they remain unchanged under the Ricci flow - at least up to rescaling. There are many natural questions to ask about the behaviour of the Ricci flow starting at, or close to, Einstein metrics (See e.g. \cite{KK}). In the boundary case, however, condition (\ref{wrongdirichlet}) is not compatible with the rescaling behavior of Einstein metrics, under Ricci flow:
$$g_t=(1-2\mu t)g_0$$
 since this would include rescaling the boundary as well. An easy way to get around this problem, is to fix instead the conformal class of the induced metric on the boundary. The problem now, is that this will only give us a system of $\frac{1}{2}n(n-1)-1$ equations, while Ricci flow is defined by a system of $\frac{1}{2}n(n-1)$ equations. We will therefore have to supplement the conformal fixing with one more boundary equation.\\

It is natural then to introduce a geometric Neumann-type boundary condition, perhaps in terms of the second fundamental form of the metric. However, since we only need one more equation, we could simply choose its trace, the mean curvature. It has been shown by Michael Anderson \cite{And2} that the boundary map 
$$\Pi:\mathcal{E}^{m,\alpha}\to \mathcal{C}^{m,\alpha}\p{\PS}\times C^{m-1,\alpha}\p{\PS}\qquad \qquad g\mapsto \p{\sq{g^T},\MC_g}$$
is indeed a smooth Fredholm operator of index $0$. Here $\mathcal{E}^{m,\alpha}$ denotes the space of $C^{m,\alpha}$ Einstein metrics modulo diffeomorphisms that fix the boundary, $\mathcal{C}^{m,\alpha}$ the space of boundary metrics modulo conformal class, and $\MC_g$ the mean curvature of the boundary. This implies that the Einstein equation (\ref{Einsteineq}) with boundary conditions defined by $\Pi$ is a well-posed elliptic boundary value problem. Using the same conditions for the Ricci flow, was studied by Panagiotis Gianniotis in \cite{PG}. That is, he considered the system 
$$\begin{cases} \partial_t g=-2\Ric_g & \text{in } M\times [0,T]\\
\Pi(g)=\p{\sq{\gamma},\eta} & \text{on }\PS \times [0,T]\end{cases} 
$$
for some Riemannian metric $\gamma$ and a function $\eta$, both defined on $\partial M\times [0,T]$. In this paper, we will attempt to choose conditions that only depend on the metric itself. This choice will be informed by a variational perspective on the Ricci flow, and we will impose on (\ref{RFIVP}) the boundary conditions
\begin{equation}\label{BC}\begin{cases}
\partial_t g_t^T=\frac{1}{n-1}\trace_{g_t^T}\p{\partial_t g_t^T} g^T\\
\partial_t\MC_{g_t}=F\p{g_t^T,\partial_t g_t^T}
\end{cases}\end{equation}
for some expression $F$, since we would still have the freedom of choosing how exactly the mean curvature should evolve along the Ricci flow.\\ 

The first half of this paper is dedicated to the argument, that the most natural way for the mean curvature to evolve is by
\begin{align}\label{mainequation}\partial_t \MC_{g_t}=-\frac{1}{2(n-1)}\trace_{g_t^T}\p{\partial_t g_t^T}\MC_{g_t}.
\end{align}
This condition is not only invariant under parabolic rescaling $g_t\mapsto rg_{r^{-1}t}$ (both sides scale by $r^{-3/2}$), but it also satisfies the expected rescaling behaviour of Einstein metrics under Ricci flow:
$$\partial_t\MC_{g_t}=\frac{\mu}{\p{1-2\mu t}^{3/2}}\MC_{g_0}=-\frac{1}{2(n-1)}\trace_{g_t^T}\p{-2\mu g_0^T}\MC_{g_t}.$$
The following will serve as an initial justification of (\ref{mainequation}), and a reminder of the role of Einstein metrics in the study of Ricci flow. We will expand on this in section \ref{lambdasection}. A usual first encounter with Einstein metrics, comes from the calculus of variations related to the Einstein-Hilbert action of a Riemannian manifold $\p{M,g}$:
$$S_g:=\int_M \scal_g \; dV_g.$$
On a manifold without boundary, the Einstein condition (\ref{Einsteineq}) is equivalent to being a critical point of $S_g$, among metrics of the same volume. In fact, Einstein metrics are minimizers of $S_g$ in their own volume-prescribed conformal class (See \cite[Theorem 4.21]{Bes}). The first variation of the Einstein-Hilbert action in the direction $h\in\Sym$, is given by
$$S'_g(h)=-\int_M \g{G_g}{h}-\delta_g\p{\delta_g h+\nabla \trace_g h}\; dV_g$$
where $G_g=\Ric_g-\frac{1}{2}\scal_g\cdot g$ is the Einstein tensor and $\delta_g$ denotes divergence. The first variation formula can e.g. be found in \cite[Proposition 4.17]{Bes}. Usually, we would let Stokes' theorem take care of the divergence term, however in the case of a manifold with boundary, one has: 
$$S'_g(h)=-\int_M \g{G_g}{h}\; dV_g-\int_{\PS}\g{\delta_g h+\nabla \trace_g h}{\nu}\; dA_g$$
where $\nu$ is the outwards pointing unit normal field. The boundary term
\begin{equation}\label{Done}\Done(h):=\g{\delta_g h+\nabla\trace_g h}{\nu}\end{equation}
shall play an important role in this article, as it also crops up in the first variation of Perelman's $\lambda$-functional:
\begin{equation}\label{ldef}\lambda_g:=\inf_{f\in C_g^\infty\p{M}}\int_{M}\p{\scal_g+|\nabla f|_g^2}e^{-f}\; dV_g\end{equation}
	where $C_g^\infty\p{M}$ is the space of smooth functions satisfying $\int_M e^{-f}\; dV_g=1$.
In the next section (specifically, in Proposition \ref{D1eq}), we will show that for a variation $h$, satisfying (\ref{BC})-(\ref{mainequation}), the first variation $\lambda'_g(h)$ admits no boundary term.\\

The reader might recognize $\Done$ as the orthogonal complement of the momentum constraint equation for Einstein metrics (also known as the Codazzi-Mainardi equation). That is, for any tangential vector field $\tau$, the second fundamental form $\SF_g$ satisfies:
$$\g{\delta_g\SF_g+\nabla \trace_{g}\SF_g}{\tau}=\Ric_g\p{\tau,\nu}=0$$
More relevant, perhaps, is the role it plays as the integrand in the definition of ADM-mass for asymptotically flat metrics:
$$m_{ADM}(g)=-\frac{1}{16\pi}\lim_{r\to \infty}\int_{S_r}D_{\hat{g}}(g)\; dA_{\hat{g}}$$
where $\hat{g}$ denotes the Euclidean metric and $S_r$ is a geodesic sphere. \\

As an argument for the soundness of the proposed boundary conditions, we prove the following theorems.

\begin{theorem}[Existence and uniqueness] Let $g$ be a smooth compatible Riemannian metric, then, for $p>n$, there exist a unique solution to the initial boundary value problem
\begin{equation}\label{Ricciflowintro}
\begin{cases}
\partial_t g_t=-2\Ric_{g_t} & \text{in } M\times [0,T]\\
\begin{rcases}
	\p{\partial_t g_t}^T=\frac{1}{n-1}\trace_{g_t^T}\p{\partial_t g_t^T}g_t^T\\
	\partial_t\MC_{g_t}=-\frac{1}{2(n-1)}\trace_{g_t^T}\p{\partial_t g_{t}^T}\MC_{g_t}
\end{rcases} &\text{on } \partial M\times [0,T]\\
\left.g_{t}\right|_{t=0}=g
\end{cases}
\end{equation}	
Such that $\partial_t^{m}g_t\in L^p\p{M\times [0,T]}$ for $0\leq m\leq 2$ and $g_t\in W^{4,p}\p{M}$ for all $t\in [0,T]$.
\end{theorem}

For an initial metric to be compatible with the Ricci flow initial boundary value problem, we require that it is possible for the system (\ref{Ricciflowintro}) to start at time $0$. In particular, this requires that the zeroth and first order initial time derivatives of the flow equation, satisfy the correspondingly differentiated boundary equations. Thus, for example, we need the initial Ricci curvature tensor to satisfy
\begin{align*}\Ric_g^T=&\frac{1}{n-1}\trace_{g^T}\p{\Ric_g^T}g^T\\ \p{\Delta_L \Ric_g}^T=&\frac{1}{n-1}\p{\trace_{g^T}\p{\p{\Delta_L\Ric_g}^T}+2\abs{\Ric_g^T}_{g^T}^2-2\trace_{g^T}\p{\Ric_g^T}^2}g^T\end{align*}
 and similarly for the second boundary condition. The Lichnerowicz Laplacian shows up as the result of the evolution equation for the Ricci tensor under the Ricci flow \cite[Theorem 1.174]{Bes}:
$$\left.\partial_t \Ric_{g_t}\right|_{t=0}=\Delta_L\Ric_g$$
Higher order compatibility conditions are defined analogously and we refer to Definition \ref{conditions} for a precise statement of the conditions.

\begin{theorem}[Boundary regularity] Let $g_t$ be a solution to the Ricci flow initial boundary value problem (\ref{Ricciflowintro}) for a smooth initial metric $g_0$, satisfying compatibility conditions up to order $2k$, then $g_t$ is of H\"older class $C^{2k+1,\alpha}$ up to the boundary.
\end{theorem}
The exact statements can be found in Theorems \ref{existanduniq} and \ref{compandreg}, respectively.
These results are closely linked to \cite[Theorem 1.2 and 1.3]{PG}, the main difference being the order of the boundary conditions. While Gianniotis prescribes the mean curvature itself, we choose to prescribe the evolution of it. Note that for both types of boundary conditions, existence and uniqueness of the Ricci-deTurck flow only requires zeroth order compatibility of the initial and boundary conditions. However, when it comes to the Ricci flow, prescribing the evolution of the mean curvature will require an additional degree of compatibility, to obtain uniqueness of the solution. We show that Einstein metrics naturally satisfy the necessary conditions (See Remark \ref{compatiblemetric}).\\

The paper is organised as follows: After discussing how the variational perspective informs our choice of boundary conditions (Sections 2 and 3), we shall proceed, in Section 4, with setting the stage for solving parabolic initial boundary value problems on compact manifolds with boundary. We will, in particular, introduce the relevant function spaces and relay relevant results from the literature, specifically \cite{Sol1}. Section 5 will center on solving the (strictly) parabolic analogue of Ricci flow - Ricci deTurck flow. This will allow us, in the last section, to prove the aforementioned theorems on existence and boundary regularity of the Ricci flow initial boundary value problem (\ref{Ricciflowintro}).\\

\textbf{Aknowledgement:} The author would like to express his gratitude towards his advisor K. Kr\"oncke, for his advice and encouragement. He would also like to thank L. Yudowitz for his careful reading and insightful comments.

\pagebreak

\section{Preliminaries}

As should be expected, we begin by outlining our notational and definitional conventions. We particularly stress the sign convention for the divergence, as this usually varies from author to author. Throughout, $M$ will denote a compact Riemannian $n$-manifold with boundary

\begin{Notation} For a tensor field $h\in T^{(k,l)}M$ and Riemannian metric $g\in \mathcal{M}(M)$, the space of metrics on $M$, we adhere to the following sign convention:
$$\delta_g h=-\trace_{12}\nabla_g h,\qquad\qquad \Delta_g h=\delta_g\nabla_g h\qquad \text{and}\qquad \beta_g=\delta_g+\frac{1}{2}\nabla \trace_g$$
and we note that we will omit writing the musical isomorphism of the divergence $\p{\delta_g h}^\sharp$ of a $(0,2)$-tensor field $h$, when it is clear that it should be considered as a vector field. We use the same symbol $\nabla$ for all derivations naturally induced by the connection, including the gradient of functions. \\

For a vector field $X\in \mathfrak{X}(M)$, and symmetric $2$-tensors $h,k\in \Sym$, we define the following contraction:
$$h\cdot X:=h(X,\cdot)^\sharp$$ 

We will make use of the co-divergence operator $\delta^*_g:\Omega^1(M)\to \Sym$, defined as the symmetrization of the induced connection (See \cite[Definition 1.59]{Bes}). It satisfies the relation
\begin{equation}\label{codivlie}\delta_g^* \alpha=\frac{1}{2}\Lie_{\alpha^\sharp}g\end{equation}

Let $\nu $ be the outward pointing unit normal field along $\PS$, then the second fundamental form and mean curvature of $g$ is given by
$$\SF_g\p{X,Y}=g\p{\nabla_X \nu,Y}\qquad \text{and}\qquad \MC_g=\trace_{g^T}\SF_g .$$
Though it would formally require an extension of $\nu$ away from the boundary, it is often advantageous to view $\SF_g$ as a Lie derivative: 
$$\SF_g=\frac{1}{2}\Lie_\nu g\qquad \text{and}\qquad \MC_g=\frac{1}{2}\trace_{g^T}\Lie_\nu g.$$
It is possible to show that this representation is independent of the chosen extension. It also gives credence to the view, that a condition on the second fundamental form is a "geometric Neumann condition".\\   

We will stick to Gianniotis' indexing practice of using latin indices for coordinates in the interior of $M$, while using Greek indices for the boundary.\\

In the latter sections of the paper, we will rely heavily on Hamilton's $\star$-notation (see \cite{Ham} and \cite{CLN}): For tensors $A,B$, we use $A\star B$ to denote a linear combination of contractions of $A\otimes B$.
\end{Notation}

As was described in the introduction, the boundary conditions we will impose on the Ricci flow shall be: the preservation of the conformal class of the induced metric on the boundary, and a prescribed evolution of the mean curvature. For completeness, we therefore include the derivation of the first variation of the mean curvature:

\begin{lemma} \label{MCfirstvar}
Let $h\in \Sym$ be the directional tensor field of a variation preserving the conformal class of $g^T$. Then the first variation of the mean curvature is given by the formula
	\begin{equation}\label{firstvarMC}\MC_g'(h)=\frac{1}{2}\trace_{g^T}(\nabla_\nu h)+\delta_{g^T}(h\cdot \nu)-\frac{1}{2}h(\nu,\nu)\MC_g\end{equation}
	where $\delta_{g^T}\p{h\cdot \nu}$ denotes the divergence of $(h\cdot \nu)^T$ with respect to the metric $g^T$
\end{lemma}
\begin{proof}
Let $g_t =g+th$, then, since $h$ preserves the conformal class of $g^T$, we can choose a stationary local orthonormal basis $\tb{e_\alpha}_\alpha$ for the tangential part of the tangent space at the boundary, such that $g_t\p{e_\alpha,\nu}\equiv 0$. Consequently 
$$\left.\partial_t\right|_{t=0}g_t\p{\nu,e_\alpha}=g\p{\nu',e_\alpha}+h\p{\nu,e_\alpha}\qquad \qquad \left.\partial_t\right|_{t=0}g_t\p{\nu,\nu}=2g\p{\nu',\nu}+h\p{\nu,\nu}$$	
from which it follows that 
$$\nu'=-\p{h\cdot \nu}^T-\frac{1}{2}h\p{\nu,\nu}\nu.$$
Since $h^T$ must be a scalar function multiple of the boundary metric, say $h^T=\varphi g^T$, we have 
$$\left.\partial_t\right|_{t=0}2\MC_{g_t}=\trace_{g^T}\p{\Lie_\nu h+\Lie_{\nu'}g}-\varphi\trace_{g^T}\Lie_\nu g=\trace_{g^T}\p{\nabla_{\nu}h}+2\delta_{g^T}\p{h\cdot \nu}-\frac{1}{2}h\p{\nu,\nu}\trace_{g^T}\p{\Lie_\nu g}$$
since $\trace_{g^T}\p{\Lie_\nu h}=\trace_{g^T}\p{\nabla_\nu h}+\varphi\trace_{g^T}\p{\Lie_\nu g}$.
\end{proof}

As indicated in the preceding proof, we shall usually represent the first boundary condition on the variational tensor $h:=\left.\partial_t g_t\right|_{t=0}$, by the equation $h^T=\varphi g^T$, for a function $\varphi:\partial M\to \R$. We will reserve the symbol $\varphi$ for this exact application, and - when relevant - we will use the equivalent formulation 
$$\varphi_g(h)=\frac{1}{n-1}\trace_{g^T}\p{h^T}$$
With only this assumption, the boundary term of the first variation of the Einstein-Hilbert action (\ref{Done}), provides another indication of why we should choose our second boundary condition to be a prescribed evolution of the mean curvature:

\begin{proposition}\label{D1eq} Suppose $h^T=\varphi g^T$ at the boundary. Then
$$\Done(h)=2\MC_g'(h)-\delta_{g^T}(h\cdot \nu)+\varphi\MC_g$$
In particular
\begin{equation}\label{DgIntegral}\int_{\partial M}\Done(h)\; dA_g=\int_{\PS}2\MC_g'(h)+\varphi\MC_g\; dA_g\end{equation}
\end{proposition}
\begin{proof} Recall from (\ref{Done}), that $D_g(h)=\g{\delta h}{\nu}+\trace_g\nabla_\nu h$. By expanding the first term, obtain
$$\g{\delta h}{\nu}=\delta_{g^T}(h\cdot \nu)-h(\nu,\nu)\MC_g-\nabla_\nu h(\nu,\nu)+\varphi \MC_g$$
Adding the second term yields
\begin{align*}\Done(h)=\g{\delta_g h}{\nu}+\trace_g \nabla_\nu h=\delta_{g^T}(h\cdot \nu)-h(\nu,\nu)\MC_g+\trace_{g^T}(\nabla_\nu h)+\varphi\MC_g\end{align*}
Lastly, we apply the first variation of the mean curvature (\ref{firstvarMC}).
\end{proof}

It is in fact possible to study the critical metrics of the Einstein-Hilbert action, even in the light of this boundary term. One possibility is to assume that the variation preserves a minimal boundary, so that the terms of the integrand in (\ref{DgIntegral}) vanish separately. It has recently been shown by Akutagawa \cite{Aku}, that this situation still allows for an Obata-type theorem. Just like for manifolds without boundary, it states that Einstein metrics are the unique minimizers of the Einstein-Hilbert action in their conformal class - at least up to rescaling.\\

Another possibility is to amend the Einstein-Hilbert Action:
$$\int_M \scal_g\; dV_g+\frac{2}{n-1}\int_{\PS}\MC_g\; dA_g$$
This was first introduced by J. F. Escobar in \cite{Esc} and studied extensively by H. Araújo in \cite{Ara}. The advantage of this functional, is that the second part of the integrand in (\ref{DgIntegral}) cancels, seemingly allowing metrics of non-minimal boundary. However, as Araújo proves, if one does not assume mean curvature preservation, the only critical points are Einstein metrics with totally geodesic boundary (at least under the usual volume-preserving assumption). What then, if we additionally assume that the variation preserves the mean curvature of the boundary? Then - unfortunately - it can be shown that we lose the aforementioned property of an Obata-type theorem. I.e. Einstein metrics are no longer the unique minimizers in their conformal class.

\pagebreak

\section{The $\lambda$-functional}\label{lambdasection}

We now return to the subject at hand: Ricci flow on manifolds with boundary. It is well attested, that Ricci flow is not a gradient flow in the strict sense. But Perelman discovered in \cite{Per}, that it can in fact be studied as if it was a gradient flow for the $\lambda$-functional:

\begin{definition}[The $\lambda$-functional] Given a compact Riemannian manifold $M$ with boundary, we define the functional
\begin{equation}\label{lambdadef}\lambda_g:=\inf_{f\in C^\infty_g\p{M}}\int_{M}\scal_g+|\nabla f|_g^2\; dW_g\end{equation}
where $dW_g:=e^{-f}dV_g$ and $ C^{\infty}_g\p{M}$ is the space of smooth functions on $M$, satisfying the volume constraint  
\begin{equation}\label{e-f}\int_{M}1\; dW_g=\int_M e^{-f}\: dV_g=1\end{equation}
\end{definition}

\begin{lemma}\label{minimizer} For a given smooth Riemannian metric $g$, a unique minimizer $f_g\in C^\infty_g\p{M}$ in (\ref{lambdadef}) exists, and it satisfies 
\begin{equation}\label{fproperties}\begin{cases}
-2\Delta_g f_g-|\nabla f_g|_g^2+\scal_g=\lambda_g & \text{in } M\\ \nabla_\nu f_g=0& \text{on }\partial M \end{cases}\end{equation}
\end{lemma}
\begin{proof} Consider the Schr\"odinger operator $L:=4\Delta_g+\scal_g$. We claim that $\lambda_g$ is the principal Neumann eigenvalue of $L$, denoted $\lambda_1(L)$. To see this, take the variational formula for $\lambda_1(L)$:
\begin{equation*}
\lambda_1(L)=\inf_{\scriptsize{\begin{matrix}w\in C^\infty(M)\\ \norm{w}_{L^2}=1\end{matrix}}}	\int_M w Lw\; dV_g=\inf_{\scriptsize{\begin{matrix}w\in C^\infty(M)\\ \norm{w}_{L^2}=1\end{matrix}}}\tb{\int_M 4\abs{\nabla w}^2+\scal_g w^2\; dV_g}
\end{equation*}
which is equivalent to (\ref{lambdadef}) by a transformation $f=-2\ln \abs{w}$. By standard theory (See e.g. \cite[Theorem 8.38]{GT}), there exists a unique positive eigenfunction $w_g$, with $\norm{w_g}_{L^2(M)}=1$, such that 
$$Lw_g=\lambda_g w_g.$$
Set $f_g=-2\ln w_g$, then 
$$\lambda_g w_g=Lw_g=\p{-2\Delta f_g-\abs{\nabla f_g}^2+\scal_g}w_g.$$
As $w_g$ is positive and since $f_g$ inherits the Neumann boundary condition of $w_g$, we have proven both existence and (\ref{fproperties}). 
\end{proof}

We have introduced this functional, so that it may help inform us of what behaviour we should expect of the Ricci flow at the boundary. As we did for the Einstein-Hilbert action, we study its first variation:

\begin{proposition}[First variation of $\lambda$]\label{firstvarlambda} Let $h\in \text{Sym}^2 \p{M}$ be the direction of a weighted-volume preserving variation, that satisfies the boundary condition $h^T=\varphi g^T$, for a $\varphi\in C^\infty\p{\partial M}$. Then
\begin{equation}\label{firstvarlambdaeq}\lambda'_g(h)=-\int_{M}\g{\Ric_g+\nabla^2f_g}{h}\; dW_g-\int_{\partial M}2\MC_g'(h)+\varphi\MC_g\;dB_g\end{equation}
where $dW_g:=e^{-f_g}dV_g$ and $dB_g:=e^{-f_g}dA_g$.
\end{proposition}

\begin{proof}
Set $g_t:=g+th$ and $f_t:=f_{g_t}$. Volume preservation implies
$$0=\ddt \int_{M}1\; dW_{g_t}=\int_{M}\p{\frac{1}{2}\trace_g h-f_g'}\; dW_g.$$	
By the Leibniz rule
\begin{align*}\ddt \lambda_{g_t}=&\int_{M}\left[\ddt \p{\scal_{g_t}+g_t\p{\nabla f_t,\nabla f_t}}\right]\; dW_g +\int_{M}\p{\scal_g+|\nabla f_g|_g^2}\p{\frac{1}{2}\trace_g h-f_g'}\; dW_g\end{align*}
Recalling the first variation of the scalar curvature, we have
\begin{align*}\int_M \left[\ddt\scal_{g_t}\right]\;dW_g=&\int_{M}-\g{\Ric_g}h+\Delta\trace_g h+\delta_g(\delta_g h)\; dW_g.\end{align*}
Integrating by parts with respect to $dV_g$,
\begin{align*}\int_{M}\delta_g(\delta_g h)\; dW_g=&-\int_M \g{\delta_g h}{\nabla f_g}\; dW_g-\int_{\partial M}\g{\delta_g h}\nu\; dB_g \\ =& \int_{M}h\p{\nabla f_g,\nabla f_g}-\g{\nabla^2 f_g}h\; dW_g-\int_{\partial M}\g{\delta_g h}\nu-h\p{\nabla f_g,\nu}\; dB_g.\end{align*}
As for the variation of the second part of the integrand of $\lambda$, we calculate
\begin{align*}\ddt |\nabla f_t|_{g_t}^2=-h\p{\nabla f_g,\nabla f_g}+2\g{\nabla f_g'}{\nabla f_g}. \end{align*}
Again, integrating by parts
\begin{align*}\int_{M} \Delta_g \trace_g h+2\g{\nabla f'_g}{\nabla f_g} \; dW_g=&\int_{M}-2 \p{\Delta_g f_g+|\nabla f_g|^2_g}\p{\frac{1}{2}\trace_g h-f_g'}\; dW_g\\ &-\int_{\partial M}\nabla_\nu\trace_g h+2\p{\frac{1}{2}\trace_g h-f_g'}\nabla_\nu f_g\; dB_g.\end{align*}
All in all: 
\begin{align*}\ddt \lambda_{g_t}=& -\int_M \g{\Ric_g+\nabla^2 f_g}h\; dW_g+ \int_{M}\p{-2\Delta_g f_g-|\nabla f_g|_g^2+\scal_g}\p{\frac{1}{2}\trace_g h-f_g'}\; dW_g\\ &-\int_{\partial M} \g{\delta_g h}\nu+\nabla_\nu\trace_g h-h\p{\nabla f_g,\nu}\; dB_g-2\int_{\partial M}	\p{\frac{1}{2}\trace_g h-f_g'}\nabla_\nu f_g\; dB_g.
\end{align*}
From Lemma \ref{minimizer}, we may conclude that the second and fourth terms vanish. By Proposition \ref{D1eq} and Stokes theorem
\begin{align*}\int_{\partial M}\Done(h)-h\p{\nabla f_g,\nu}\; dB_g &=\int_{\partial M}2\MC_g'(h)+\varphi\MC_g-\delta_{g^T}\p{h\cdot \nu}-h\p{\nabla f_g,\nu}\; dB_g\\ &=\int_{\partial M}2\MC_g'(h)+\varphi\MC_g\; dB_g\end{align*}
since 
$$\p{\delta_{g^T}\p{h\cdot \nu}+h\p{\nabla f_g,\nu}}\; dB_g=\delta_{g^T}\p{e^{-f_g}h\cdot \nu}\; dV_g$$
\end{proof}

Thus, exactly as in the Einstein-Hilbert case, mean curvature preservation would force minimality of the boundary - at least if we want our dear Einstein metrics to be critical points:

\begin{corollary}\label{lambda'=0} If $g$ is Einstein and $\MC_g\equiv 0$, then $\lambda_g'(h)=0$ for all volume preserving (equiv. weighted-volume preserving) $h\in \text{Sym}^2\p{M}$, satisfying $h^T=\varphi g^T$ and $\MC_g'\p{h}=0$.
\end{corollary}
\begin{proof}
Let $\mu$ be the Einstein constant of $g$, then
$$\int_{M}\scal_g+|\nabla f|_g^2\; dW_g=n\mu+\int_M|\nabla f|_g^2\; dW_g $$
is minimized by a constant $f_g$. Thus
$$\lambda_g'(h)=-\int_M \g{\Ric_g}h \; dW_g=-\mu e^{-f_g}\int_M \trace_g h\; dV_g=0$$
for volume-preserving $h$.
\end{proof}

Existence of a Ricci flow of metrics with vanishing second fundamental form at the boundary has recently been studied by \cite{Chow}. These are flows where each side of 
 \begin{equation}\label{DgRic}\MC_g'(\Ric)=-\frac{1}{2(n-1)}\trace_{g^T}\p{\Ric}\MC_g\qquad \qquad \text{on} \quad \partial M\end{equation}
 vanish separately. These are of course interesting in their own right, but are very restrictive on the choice initial metric.\\

There exist several related functionals in the field surrounding Ricci flow and Einstein metrics. Whenever the functional is build around a scalar curvature term, it is not unnatural to expect to find $\Done\p{h}$ in the integrand of the boundary term of the first variation - as is the case for $S_g$ and $\lambda_g$. In fact, it might be pertinent to emphasize that:

\begin{remark} After fixing the conformal class of the boundary, the following functionals all have the same first order boundary term (up to a constant)
$$S_g=\int_M\scal_g \; dV_g\qquad \qquad\text{Einstein-Hilbert action}$$
$$\lambda_g=\inf_{f\in C_g^\infty(M)}\int_{M}\scal_g+|\nabla f|_g^2\; dW_g\qquad \qquad \text{Perelman's $\lambda$-functional}$$
$$\mu_+(g)=\inf_{f\in C_g^\infty(M)}\frac{1}{2}\int_{M}\scal_g +|\nabla f|_g^2 -2f\; dW_g\qquad \qquad \text{The Expander Entropy}$$
$$\mu_-(g)=\inf_{f\in C_{\tau,g}^\infty(M)}\frac{1}{\p{4\pi\tau}^{n/2}}\int_{M}\tau\p{\scal_g +|\nabla f|_g^2} +f-n\; dW_g\qquad \qquad \text{The Shrinker Entropy}\footnote{Here the infimum is taken over smooth functions with $\int_M e^{-f}\; dV=\p{4\pi \tau}^{n/2}$.}$$
Both entropies are non-decreasing under the Ricci flow, and the boundary term  in the second remains unchanged if we allow for variable $\tau$. For the application of the entropies in the study of stability of Einstein metrics under Ricci flow, we refer you to \cite{KK1}.
\end{remark}

Having now established a set of agreeable boundary conditions, we dedicate the rest of the paper to proving that these conditions leads to adequate short-term existence and uniqueness of Ricci flow on manifolds with boundary.
\pagebreak

\section{Function Spaces and Boundary Value Problems}

We will apply theory of boundary value problems for linear parabolic systems, developed - in part - by V. A. Solonnikov, which can be found his book \cite{Sol1}. We therefore follow his conventions with regards to norms and function spaces. In Definition \ref{fctspacesonM} we will describe how these spaces can be extended to metrics and other tensors.

\begin{definition}[Function spaces] For a function $f:\R^n\times [0,T]\to \R$, we define the following semi-norms:
$$\left\langle f\right\rangle_{(l)}:=\sum_{2m+|\alpha|=\lfloor l\rfloor}\sup_{x,y,t}\frac{\abs{\partial_t^m D_x^\alpha f(x,t)-\partial_t^m D_x^\alpha f(y,t)}}{\abs{x-y}^{l-\lfloor l\rfloor}}+\sup_{x,t,s}\frac{\abs{\partial_t^m D_x^\alpha f(x,t)-\partial_t^m D_x^\alpha f(x,s)}}{\abs{t-s}^{\frac{l-\lfloor l\rfloor}2}}$$
and, for $0<\varepsilon<1$,
$$\left[f\right]_{p,\varepsilon}:=\p{\int_{0}^T\int_{\R^n}\int_{\R^n}\frac{\abs{f(x,t)-f(y,t)}^p}{\abs{x-y}^{n+\varepsilon p}}\; dx\; dy\; dt}^{\frac{1}{p}}+\p{\int_{\R^n}\int_0^T\int_0^T\frac{\abs{f(x,t)-f(x,s)}^p}{\abs{t-s}^{1+\frac{\varepsilon p}{2}}}\; dt\; ds\;dx}^{\frac{1}{p}}$$

\begin{enumerate}[i)]
\item For $l>0$, $C^{(l)}\p{\R^n\times [0,T]}$ will denote the closure of the set of compactly supported smooth functions, in the norm
$$\norm{f}_{C^{(l)}}=\left\langle f\right\rangle_{(l)}+\sum_{2m+|\alpha|=0}^{\lfloor l\rfloor-1}\sup_{x,t}\left|\partial_t^mD_x^\alpha f(x,t)\right|$$
\item For integer $l$, we will define the Sobolev space $W^{(2l),p}\p{\R^n\times [0,T]}$ as the closure of the set of $C_c^\infty$-functions, in the norm
	$$\norm{f}_{W^{(2l),p}}=\p{\sum_{2m+|\alpha|\leq 2l}\norm{\partial_t^mD_x^\alpha f}^p_{L^p}}^{\frac{1}p}$$
\item For non-integer $l$, we introduce $B^{(l),p}\p{\R^{n-1}\times [0,T]}$ as the closure of $C_c^\infty$ in the norm
$$\norm{f}_{B^{(l),p}}=\sum_{2m+|\alpha|=\lfloor l\rfloor}\left[\partial_t^mD_x^\alpha f\right]_{p,l-\lfloor l\rfloor}+\sum_{2m+|\alpha|=0}^{\lfloor l\rfloor-1}\norm{\partial^m_t D_x^\alpha f}_{L^p}$$
The choice of the symbol $B$ and space $\R^{n-1}$, reflects the fact that the use of these norms are restricted to the boundary.
\end{enumerate}
\end{definition}

\begin{definition}[Norms on tensors]\label{fctspacesonM} Fix a finite family of coordinate neighbourhoods $\tb{U_m}_{m=1}^N$, covering $M$. Add to this another cover of slightly smaller sets $\tb{V_m}_{m=1}^N$, such that $\overline{V_m}\subset U_m $ for all $m$. Let $\tb{\eta_m}_{m=1}^N$ be a partition of unity, subordinate to the smaller cover $\tb{V_m}_{m=1}^N$. To each pair $(V_m,U_m)$ we associate a smooth cut-off function, $\chi_m$, with support in $U_m$, such that $V_m\subset \tb{p\in U_m\; : \; \chi_m(p)=1}$. We construct the coordinate functions $\psi_m:U_m\to \R^n$ such that
\begin{itemize} 
\item If $U_m\subset M^\circ$, $\psi_m$ is a diffeomorphism onto $\R^n$	
\item If $U_m\cap \partial M\neq \emptyset$, $\psi_m$ is a diffeomorphism onto $\R^n_+ :=\tb{x\in \R^n\; :\; x_n\geq 0}$
\end{itemize}
Given a tensor $\omega\in T^{(a,b)}M$, we denote its local representation on $U_m$ as $\hat{\omega}_m:=\p{\psi_m}_\ast \p{\chi_m \omega}$. If $\omega(t)$ is a one-parameter family of tensors, defined for $t\in [0,T]$, we extend the norms $C^{(l)}$ and $W^{(2l),p}$ to $M_T=M\times [0,T]$:
$$\norm{\omega}_{C^{(l)}\p{M_T}}:=\sum_{m=1}^N \norm{\hat{\omega}_m}_{C^{(l)}\p{\R^n_{(+)}\times [0,T]}}=\sum_{m=1}^N\sum_k \norm{\hat{\omega}_m^k}_{C^{(l)}\p{\R^n_{(+)}\times [0,T]}}$$
with $\hat{\omega}_m^k$ denoting the $k$'th coordinate function of $\hat{\omega}_m$ and $\R^n_{(+)}$ equal to $\R^n$ or $\R^n_+$, depending on whether $U_m$ is interior or not.  We extend $W^{(2l),p}$ analogously. For tensors over $\partial M$, we use that $\left.\psi_m\right|_{\partial M\cap U_m}$ is a diffeomorphism to $\R^{n-1}$, such that we can extend the $B^{(l),p}$ norm to $\partial M_T:=\partial M\times [0,T]$:
$$\norm{\omega}_{B^{(l),p}\p{\partial M_T}}:=\sum_{m=1}^{N'} \norm{\hat{\omega}_m}_{B^{(l),p}\p{\R^{n-1}\times [0,T]}}$$
where the summation is only over boundary charts.
\end{definition}

In certain cases, we will stress the type of tensor, writing $W^{(2l),p}\p{M_T,\Sym}$ to denote the space of evolving symmetric $2$-tensors over $M$, with bounded $W^{(2l),p}$-norm. On occasion, we will suppress the base manifold in the notation, when it is clear what is meant. It is worth noting, that $B^{(l),p}$ is solely defined on $\PS_T$, while $W^{(2l),p}$ is only relevant on $M_T$. \\

\begin{remark} The spaces $C^{(l)}\p{M_T}$ and $W^{(2l),p}\p{M_T}$, by dint of their definition on $\R^n\times [0,T]$, conform to the usual relationship defined by the Sobolev embedding theorem and H\"older inequalities. Similarly, the spaces $B^{(l),p}\p{\partial M_T}$ and $C^{(l)}\p{\PS_T}$ are related by a H\"older-type inequality:
\begin{equation}\label{Hoelder}\norm{fg}_{B^{(l),p}}\leq C\norm{f}_{C^{\p{\lfloor l \rfloor}}}\norm{g}_{B^{(l),p}}\end{equation}
The relationship between the spaces $B^{(l),p}\p{\PS_T}$ and $W^{(2l),p}\p{M_T}$ is governed by the following theorem. 	
\end{remark}

\begin{theorem}[Solonnikov embedding]\cite[Theorem 5.1]{Sol1}  \label{SolEmb} Let $\omega\in W^{(2k),p}\p{M_T}$ and $2m+l< 2k-\frac{1}{p}$. If $\nabla$ is the Levi-Civita connection for an arbitrary smooth stationary metric $g$, then
$$\norm{\partial_t^m \nabla^l\omega}_{B^{\p{2k-2m-l-\frac{1}{p}},p}\p{\partial M_T}}\leq C\p{g}\norm{\omega}_{W^{(2k),p}\p{M_T}} $$
where the constant depends on $g^{-1}$ and the first $l$ derivatives of $g$.\end{theorem}

At this point, we will concede to using `constants' that are allowed to change from line to line, but we will always strive to specify any and all relevant dependencies of these. Solvability of the parabolic PDE's we are interested in, is dependent on the following two properties.

\begin{definition}[Parabolicity and complementarity]\label{conditions} Let $\Omega$ be a domain in $\R^n$. For a given set of vector-valued functions, $f:\Omega \times [0,T]\to \R^m$, $\Phi:\partial\Omega \times [0,T]\to \R^r$ and $\psi:\Omega \to \R^m$, we consider the initial boundary value problem:
\begin{equation}\label{paracompl}
	\begin{cases}
	\mathscr{L}u=f & \text{in } \Omega\times [0,T]\\
	\mathscr{B}u=\Phi & \text{on }\partial \Omega \times [0,T] \\
	u=\psi	 & \text{for } t=0
	\end{cases}
\end{equation}
Here, $\mathscr{L}$ and $\mathscr{B}$ denotes matrices of linear partial differential operators with entries $l_{ij}\p{\pp{x},\dt}$ respectively $b_{kj}\p{\pp{x},\dt}$, where $1\leq i,j\leq m$ and $1\leq k\leq r$. Each entry can be regarded as a function on $\R^n\times \mathbb{C}$ with coefficients that might depend on $(x,t)\in \Omega \times [0,T]$.  Let $s_i$ and $t_j$ be non-negative integers such that for any $\xi=\p{\xi_1,\ldots,\xi_n}\in \R^n$, the polynomial degree of the entries of $\mathscr{L}$, in the variable $\lambda$, satisfies
	$$\deg_\lambda l_{ij}\p{\xi\lambda,\lambda^2}\leq t_j-s_i$$
chosen\footnote{The choice of $t_j$ and $s_i$ will in general not be unique, and the complementarity condition might be satisfied for one choice, but not another.} such that $r=\frac{1}{2}\sum_{i,j=1}^m\p{t_j-s_i}$. Furthermore, we set
$$\beta_{ij}:=\deg_\lambda b_{ij}\p{\xi\lambda,\lambda^2}\qquad \qquad \sigma_i:=\max_j\p{\beta_{ij}-t_j}$$
\begin{itemize}
\item Parabolicity: Denote by $l^0_{ij}$ the principal part of $l_{ij}\p{\xi\lambda,\lambda^2}$, i.e. the part satisfying 
$$l^{0}_{ij}\p{\xi\lambda,z\lambda^2}=\lambda^{t_j-s_i}l^{0}_{ij}\p{\xi,z}.$$
Let $\mathscr{L}_0$ denote the matrix with entries $l^0_{ij}$ and let $L\p{\xi,z}=\det \mathscr{L}_0\p{\xi,z} $. We say that $\mathscr{L}$ satisfies the parabolicity condition if the $z$-respective roots of the polynomial $L\p{i\xi,z}$ satisfy 
\begin{equation}\label{parabcond}\Real z\leq -\delta |\xi|^2\end{equation}
for some $\delta>0$.	
\item Complementarity: Denote by $b^0_{ij}$ the principal part of $b_{ij}\p{\xi\lambda,z\lambda^2}$. i.e. the part satisfying
$$b^{0}_{ij}\p{\xi\lambda,z\lambda^2}=\lambda^{\sigma_i+t_j}b^0_{ij}\p{\xi,z}.$$
We let $\mathscr{B}_0$ be the matrix with entries $b^0_{ij}$.\\
 Given a point $x\in \partial\Omega$, we let $\zeta(x)$ be a tangent vector and $\nu(x)$ the inward unit normal at $x$. If for some $0<\delta_1<\delta$, the $z$-respective roots of $L\p{i\xi,z}$ satisfy
 $$\Real z\geq -\delta_1|\zeta|^2$$ 
 then, as a polynomial in $\tau$, $L\p{i(\zeta +\tau\nu),z}$ has $2r$ complex roots, half of them, ($\tau_1^+,\ldots,\tau_r^+$), with positive imaginary part. Define the polynomial 
 $$M^+(\zeta, \tau,z)=\prod_{s=1}^r\p{\tau-\tau_s^+(\zeta,z)}$$
 We say that $\mathscr{B}$ satisfies the complementarity condition, if at every point $(x,t)\in \partial\Omega\times [0,T]$ and every tangent vector $\zeta(x)$, the rows of 
 $$\mathscr{M}\p{i(\zeta +\tau\nu),z}:=L\p{i(\zeta +\tau\nu),z}\mathscr{B}_0\p{i(\zeta +\tau\nu),z}\mathscr{L}_0^{-1}\p{i(\zeta +\tau\nu),z}$$
 are linearly independent modulo $M^+$.
\end{itemize}
We shall furthermore say that the system (\ref{paracompl}) satisfies compatibility conditions up to order $k$, if the derivatives $\left.\partial_t^mu\right|_{t=0}$ - determined uniquely by $f$ and $\psi$ - satisfy the boundary conditions
$$\left.\partial_t^m\p{\mathscr{B}_i\cdot u}\right|_{t=0}=\left.\partial_t^m\Phi_i\right|_{t=0} $$
for $m=0,\ldots ,\left\lfloor\frac{k-\sigma_i}{2} \right\rfloor$, where $\mathscr{B}_i$ denotes the $i$'th row of $\mathscr{B}$ and $\cdot$ the standard vector product. 
\end{definition}

\begin{remark} The type of system we will be interested in, has as its defining operator a diagonal matrix,
$$\mathscr{L}\p{\xi,z}=\begin{pmatrix}l_{1}(\xi,z) & &\\ & \ddots &\\ & & l_m\p{\xi,z } \end{pmatrix}$$
consisting of second order operators, $l_k\p{\lambda \xi,\lambda^2z}=\lambda^2l_k\p{\xi, z}$ for all $1\leq k\leq m$. In this particular case, the matrix $\mathscr{M}$ is the product
$$\begin{pmatrix}|&  &|\\ \mathscr{B}_0^1\p{i(\zeta +\tau\nu),z}&\cdots &\mathscr{B}_0^m\p{i(\zeta +\tau\nu),z}\\ |&& |\end{pmatrix}\begin{pmatrix}\prod_{k\neq 1}l_k\p{i(\zeta +\tau\nu),z} & &\\ & \ddots &\\ & &\prod_{k\neq m} l_k\p{i(\zeta +\tau\nu),z} \end{pmatrix}$$
where $\mathscr{B}_0^k$ denotes the $k$'th column of the principal part of the boundary operator matrix. Taken modulo $M^+$, we see that the matrix whose rows has to be linearly independent for complementarity to apply, is simply
$$\begin{pmatrix}|&  &|\\ \mathscr{B}_0^1\p{i(\zeta +\tau_1^+\nu),z}&\cdots &\mathscr{B}_0^m\p{i(\zeta +\tau_m^+\nu),z}\\ |&& |\end{pmatrix}$$
where $\tau_k^+$ denotes the root of $\tau \mapsto l_k\p{i(\zeta +\tau\nu),z}$ with $\Imag \tau_k^+>0$.
\end{remark}

The following is an adaptation of a much more general theorem; \cite[Theorem 5.4]{Sol1}. We have - in particular - specified the degree and regularity of the operators to pertain to our needs. 

\begin{theorem}\label{Sol5.4} Let $k\geq 1$ and $p>3$ be integers. Let $\Omega$ be a domain in $\R^n$ and suppose $\partial \Omega$ is at least $C^{2k+2}$. Set $\Omega_T:=\Omega\times [0,T]$ and $\partial \Omega_T:=\partial\Omega \times [0,T]$ and let $f$, $\Phi^A$, $\Phi^B$, $\Phi^C$ and $g$ be a set of vector valued functions. Consider the initial boundary value problem (IBP)
$$\begin{cases}
\mathscr{L}u=f & \text{in} \quad \Omega_T\\
\begin{rcases}
\mathscr{A}u=\Phi^A\\ \mathscr{B}u=\Phi^B 	\\ \mathscr{C}u=\Phi^C
\end{rcases}& \text{on}\quad \partial \Omega_T\\
\quad u=g& \text{for}\quad t=0
\end{cases}
$$ 
where $\mathscr{L}$, $\mathscr{A}$, $\mathscr{B}$ and $\mathscr{C}$ are represented by matrices with entries $l_{ij}\p{\pp{x},\dt}$, $a_{ij}\p{\pp{x},\dt}$, $b_{ij}\p{\pp{x},\dt}$ and $c_{ij}\p{\pp{x},\dt}$, respectively. We assume that this IBP satisfies the conditions of parabolicity, complementarity and compatibility up to degree $2k-1$. Furthermore, we require the coefficients of each $l_{ij}$ be $C^{(2k)}\p{\Omega_T}$, the coefficients of $a_{ij}$, $b_{ij}$ and $c_{ij}$ be $C^{\p{2k+1}}\p{\partial \Omega_T}$, $C^{\p{2k}}\p{\partial \Omega_T}$ and $C^{\p{2k-1}}\p{\partial \Omega_T}$, repectively. We will also assume that
$$\deg_\lambda \p{l_{ij}\p{\xi \lambda,\lambda^2}}\leq 2\qquad \qquad \deg_\lambda \p{a_{ij}\p{\xi \lambda,\lambda^2}}\leq 1$$
$$\deg_\lambda \p{b_{ij}\p{\xi \lambda,\lambda^2}}\leq 2\qquad \qquad \deg_\lambda \p{c_{ij}\p{\xi \lambda,\lambda^2}}\leq 3$$
for all $i,j$.\\

If $f\in W^{(2k),p}\p{\Omega_T}$, $\Phi^{A}\in B^{\p{2k+1-\frac{1}{p}},p}\p{\partial \Omega_T}$, $\Phi^{B}\in B^{\p{2k-\frac{1}{p}},p}\p{\partial \Omega_T}$, $\Phi^{C}\in B^{\p{2k-1-\frac{1}{p}},p}\p{\partial \Omega_T}$ and $g\in W^{2k+2,p}\p{\Omega}$, then the IBP has a unique solution $u\in W^{(2k+2),p}\p{\Omega_T}$, satisfying the parabolic estimate
$$\norm{u}_{W^{(2k+2),p}}\leq C\p{ \norm{f}_{W^{(2k),p}}+\norm{\Phi^{A}}_{B^{\p{2k+1-\frac{1}{p}},p}}+\norm{\Phi^{B}}_{B^{\p{2k-\frac{1}{p}},p}}+\norm{\Phi^{C}}_{B^{\p{2k-1-\frac{1}{p}},p}}+\norm{g}_{W^{2k+2,p}}}$$
\end{theorem}

\pagebreak

\section{Ricci-deTurck Flow}\label{rdtsection}

The first result of this section provides an existence and uniqueness result for an initial boundary value problem on the upper half-plane, using the properties and results introduced in the preceding section. The subsequent theorem extends this result to a compact manifold with boundary, with a technique resembling that of A. Pulemetov \cite{AP}, where we construct an adequate approximate solution of a closely related problem. While the connection between these results and the Ricci-deTurck flow might not be immediate, it will provide the base, on which we may construct a fixed point argument to solve a Ricci-deTurck initial boundary value problem, subject to the boundary conditions introduced in Section \ref{lambdasection}.

\begin{proposition}\label{upperhalfplane} Let $p>n$ and consider the upper half plane $\Omega:=\R^n_+=\tb{\p{x_1,\ldots,x_n}\in \R^n\;:\;x_n\geq 0}$ and let $h$ be a symmetric $n\times n$-matrix with coefficients in $ C^{(2k)}\p{\Omega_T}$. We further assume that the the coefficients of $h^{-1}$ are of the same regularity, that the leading <$(n-1)\times (n-1)$-submatrix $\left.h_{\alpha\beta}\right|_{\partial\Omega_T}$ is non-degenerate, and that there exists a  $\delta >0$, such that for every $x\in \R^n_+$ and $t\in [0,T]$
\begin{equation}\label{posdef}h(x,t)^{kl}\xi_k\xi_l\geq \delta \abs{\xi}^2\end{equation}
for all vectors $\xi \in \R^n$. Let $V_1:= \R^n$, $V_2:=\text{Sym}(n-1)$, the space of symmetric $(n-1)\times (n-1)$ matrices, and $V_3:= \R$ and denote by $WB_k^p\p{\Omega_T}$, the subspace of 
\begin{equation}\label{WB}W^{(2k),p}\p{\Omega_T,\text{Sym}(n)}\oplus \bigoplus_{i=1}^3 B^{\p{2k+2-i-\frac{1}{p}},p}\p{\partial \Omega_T,V_i}\oplus W^{2k+2,p}\p{\Omega,\text{Sym}(n)}\end{equation}
such that every element $\p{f,\phi^1,\phi^2,\phi^3,\omega}\in WB_k^p\p{\Omega_T}$ satisfies the compatibility conditions up to order $2k-1$ for the system\footnote{Note that Einstein summation does not apply to the fixed index $n$}
\begin{equation}\label{reallocalsystem}
\begin{cases}
\mathscr{L}\p{\pp{x},\partial_t}v_{ij}:=	\partial_t v_{ij}-h^{kl}\partial_k\partial_l v_{ij}+\mathit{l}\p{\pp{x} v}_{ij}  \qquad \qquad\qquad\qquad\qquad \:\: = f_{ij} & \text{in}\: \Omega_T\\
\begin{rcases}A\p{\pp{x}}v_m:=-h^{kl}\p{\partial_k v_{lm}-\frac{1}{2}\partial_m v_{kl}}+a\p{\pp{x} h,v}_m & =\phi^1_m\\
B\p{\partial_t}v_{\alpha\beta}:=\partial_t v_{\alpha\beta}-(n-1)^{-1}h^{\gamma \delta}h_{\alpha\beta}\partial_t v_{\gamma\delta} & =\phi^2_{\alpha\beta}\\
C\p{\pp{x},\partial_t}v:=\p{h^{nn}}^{-\frac{1}{2}}h^{in}h^{\alpha\beta}\p{\partial_\alpha\partial_t v_{\beta i}-\frac{1}{2}\partial_i\partial_t v_{\alpha\beta}} +c\p{\pp{x} h,\partial_t v} & = \phi^3
\end{rcases} & \text{on} \: \partial \Omega_T\\
\left.v_{ij}\right|_{t=0}\qquad \qquad\qquad\qquad\qquad\qquad \quad\quad\:\:\: \:\:\:\qquad\qquad \qquad \qquad \qquad \quad\:\:\: =\omega_{ij}
\end{cases}
\end{equation}
in which the lower order terms are given by
\begin{align*}\mathit{l}\p{\pp{x} v} = -h^{kl}\p{\overline{\Gamma}_{kl}^m\partial_m v_{ij}+4\overline{\Gamma}_{li}^m\partial_k v_{mj}}-2h^{kl}v_{mj}\partial_k\overline{\Gamma}_{li}^m+ 2h^{kl}\p{\overline{\Gamma}_{kl}^m\overline{\Gamma}_{mi}^pv_{pj}+\overline{\Gamma}_{ki}^m\overline{\Gamma}_{lm}^pv_{pj}+\overline{\Gamma}_{ki}^m\overline{\Gamma}_{lj}^pv_{mp}}\end{align*}
where $\overline{\Gamma}$ is the Christoffel symbol of a fixed smooth background metric on $\Omega$, and
$$a\p{\pp{x} h,v}_m=h^{kl}h^{ij}\p{\partial_k h_{lj}-\frac{1}{2}\partial_j h_{kl}}v_{im},$$
\begin{align*}c\p{\pp{x} h,\partial_t v} =& -\frac{1}{2\sqrt{h^{nn}}}h^{\alpha\beta}h^{mn}\p{h^{ik}-\frac{h^{in}h^{kn}}{h^{nn}}}\p{\partial_k h_{m\alpha}+\partial_\alpha h_{km}-\partial_m h_{\alpha k}}\partial_t v_{\beta i}\\ & -\frac{1}{\sqrt{h^{nn}}}h^{\alpha\beta}h^{mn}\p{h^{ik}-\frac{3h^{in}h^{kn}}{2 h^{nn}}}\p{\partial_\alpha h_{\beta m}-\frac{1}{2}\partial_m h_{\alpha\beta}}\partial_t v_{mk},
\end{align*}

Then every set of input data from $WB_k^p\p{\Omega_T}$ gives a unique solution $v\in W^{(2k+2),p}\p{\Omega_T,\text{Sym}(n)}$ of (\ref{reallocalsystem}), such that 
$$\norm{v}_{W^{(2k+2),p}\p{\Omega_T}}\leq C\norm{\p{f,\phi^1,\phi^2,\phi^3,\omega}}_{WB_k^p\p{\Omega_T}}$$
\end{proposition}
\begin{proof} According to Theorem \ref{Sol5.4}, it is enough to show that (\ref{reallocalsystem}) satisfies the conditions of parabolicity and complementarity. The former is a direct consequence of assumption (\ref{posdef}), so we prove only the latter. \\

Since complementarity is a pointwise condition and is independent of the chosen basis, we fix a point $(x,t)\in\partial \R^n_+\times [0,T]$ and choose a basis for which $h_{ij}=\delta_{ij}$, the Euclidean metric. We can then choose a tangent vector $\zeta=\p{\zeta_1,\ldots, \zeta_{n-1},0}$ and a normal vector $\tau\nu=\p{0,\ldots,0,\tau}$, for $\tau \in \mathbb{C}$. Given $z\in \mathbb{C}$, satisfying $\Real z+\abs{\zeta}>0$, we consider the $\tau$-respective roots of the polynomial given by the principal part of the evolution operator:
$$\mathscr{L}_0\p{i\p{\zeta+\tau\nu},z}=z+\abs{\zeta}^2+\tau^2$$
Strictly speaking, the roots should be of the polynomial obtained from taking the determinant of the matrix version of $\mathscr{L}_0$, however, since that is simply $\p{z+\abs{\zeta}^2+\tau^2}I$, where $I$ is the $n^2\times n^2$ identity matrix, the roots are the same. Complementarity should be checked against the root with positive imaginary part:
$$\tau_0=i\sqrt{z+\abs{\zeta}^2}.$$
Taking the principal part of the boundary operators, we obtain the system of equations
\begin{equation}\label{Avlast} A\p{i\p{\zeta+\tau_0\nu}}v_\alpha=-i\p{\delta^{\gamma \beta}\zeta_\gamma v_{\beta \alpha}+\tau_0v_{n\alpha}-\frac{1}{2}\zeta_\alpha\p{\delta^{\gamma\beta}v_{\gamma\beta}+v_{nn}}}
\end{equation}
\begin{equation}\label{Avlastn} A\p{i\p{\zeta+\tau_0\nu}}v_n=-i\p{\delta^{\gamma \beta}\zeta_\gamma v_{\beta n}+\frac{1}{2}\p{\tau_0v_{nn}-\delta^{\gamma\beta}v_{\gamma\beta}}}
\end{equation}
\begin{equation}\label{Bvlast} B\p{z}v_{\alpha\beta}=z\p{v_{\alpha\beta}-\p{n-1}^{-1}\delta^{\gamma\epsilon}v_{\gamma\epsilon}\delta_{\alpha\beta}}
\end{equation}
\begin{equation}\label{Cvlast} C\p{i\p{\zeta+\tau_0\nu},z}v=iz\p{\delta^{\alpha\beta}\zeta_\alpha v_{\beta n}-\frac{1}{2}\tau_0\delta^{\alpha\beta}v_{\alpha\beta}}
\end{equation}
Complementarity is proven, if we can show that there is no non-trivial $v_{ij}$, for which (\ref{Avlast})-(\ref{Cvlast}) vanish simultaneously. Assuming, for contradiction, that such a $v$ exists, it is immediate from (\ref{Avlastn}) and (\ref{Cvlast}), that it must have $v_{nn}=0$. As usual, let us write $\varphi=\p{n-1}^{-1}\delta^{\alpha\beta}v_{\alpha\beta}$, then, upon multiplying (\ref{Avlast}) through by $\delta^{\alpha\epsilon}\zeta_\epsilon$ and using (\ref{Bvlast}), it simplifies to
$$\frac{3-n}{2}\abs{\zeta}^2\varphi+\tau_0\delta^{\alpha\epsilon}\zeta_{\epsilon}v_{n\alpha}=0$$
Together with (\ref{Cvlast}) and the definition of $\tau_0$, we can write
$$0=\frac{3-n}{2}\abs{\zeta}^2\varphi+\frac{n-1}{2}\tau_0^2\varphi=\abs{\zeta}^2\varphi+\frac{1-n}{2}\p{z+2\abs{\zeta}^2}\varphi$$
And, since $\Real z+2\abs{\zeta}^2>\abs{\zeta}^2$, the above implies $\varphi=0$, even in the borderline case of $n=3$. It immediately follows from (\ref{Avlast}) that $v_{ij}=0$ for all $i,j$.
\end{proof}

The operators $\mathcal{L}$, $A$, $B$ and $C$ in system (\ref{reallocalsystem}) were by no means chosen at random, but rather as the local expressions of the operators of a system we aim to solve on $M$. Note that the system (\ref{reallocalsystem}) takes in $7$ entries of data; a $C^{(2k)}$ matrix $h$ and a smooth background metric in the definition of the main PDE, as well as $5$ points of input data, $\p{f,\phi^1,\phi^2,\phi^3,\omega}$. The smooth background metric shall always be the local representation of the usual background metric $\overline{g}$, in whatever chart we are working with, so it will be enough to specify $h$ and the input data. As we require both $h$ and $h^{-1}$ to be $C^{(2k)}$, we can not expect that an arbitrary metric on $\R^n_+$ will do. To circumvent this problem, we interpolate the desired metric with the Euclidean, before we use it as data in the system.

\begin{theorem}\label{initialRDT} Let $k\geq 1$ and $p>n$ and fix a smooth background metric $\overline{g}$. Let $g$ be a $C^{2k+2}$ metric on a manifold $M$, and let $F\in W^{(2k),p}\p{M_T,\Sym}$ and $\Phi^i\in B^{\p{2k+2-i-\frac{1}{p}},p}\p{\partial M_T,V_i}$ for $i=1,2,3$, where $V_1=T\partial M$, $V_2=\text{Sym}^2\p{\partial M}$ and $V_3$ is the trivial line bundle over $\partial M$. Consider the initial boundary value problem
\begin{equation}\label{initialIBV}\begin{cases}
\partial_t u-\trace_g\overline{\nabla}^2u \quad \;\;=F & \text{in  }M_T\\
\begin{rcases}
\beta_g(u) \qquad\qquad  &=\Phi^1\\
\alpha_g\p{\partial_t u}&=\Phi^2\\
\MC_g'\p{\partial_t u}&=\Phi^3
\end{rcases} & \text{on  }\partial M_T\\
\left.u\right|_{t=0} \qquad  \qquad \quad\:\:=g
\end{cases}
\end{equation}
where $M_T=M\times [0,T]$ and $\partial M_T=\partial M\times [0,T]$, for some $T>0$, $\trace_g\bnab^2 =g^{kl}\bnab_k\bnab_l$, $\beta_g=\delta_g+\frac{1}{2}\nabla \trace_g$ (the Bianchi operator) and $\alpha_g\p{\partial_t u}:=\p{\partial_t u}^T-\frac{1}{n-1}\trace_{g^T}\p{\partial_t u}g^T$. Assume (\ref{initialIBV}) satisfies the compatibility conditions up to order $ 2k-1$. Specifically, $h_m:=\left.\partial_t^m u\right|_{t=0}$ - which will be uniquely determined by the initial metric and $F$ whenever $m\leq k+1$, should satisfy for $m\leq k$
\begin{align*}\begin{rcases}\beta_g(h_m) &=\left.\partial_t^m\Phi^1\right|_{t=0}\\
\alpha_g\p{h_m} &=\left.\partial_t^{m-1}\Phi^2\right|_{t=0}\\
\MC_g'\p{h_m} &=\left.\partial_t^{m-1}\Phi^3\right|_{t=0}\end{rcases} \text{on } \PS.
\end{align*}
Then (\ref{initialIBV}) is uniquely solvable in $ W^{\p{2k+2},p}\p{M_T,\Sym}$. Furthermore, the solution $u$ satisfies the parabolic estimate 
\begin{equation}\label{parabolicest}
 \norm{u}_{W^{(2k+2),p}}	\leq C\cdot \p{\norm{F}_{W^{(2k),p}} +\norm{\Phi^1}_{B^{\p{2k+1-\frac{1}{p}},p}}+\norm{\Phi^2}_{B^{\p{2k-\frac{1}{p}},p}}+\norm{\Phi^3}_{B^{\p{2k-1-\frac{1}{p}},p}}+\norm{g}_{W^{2k+2,p}}}
\end{equation}
where the constant $C>0$ depends on $g$, but not on $T$.
\end{theorem}

\begin{proof} Let $U_l$, $V_l$, $\psi_l$, $\eta_l$ and $\chi_l$ be as defined in Definition \ref{fctspacesonM}, and recall the notation $\hat{\omega}_l=\p{\psi_l}_\ast \p{\chi_l \omega}$, for the local representation of the tensor $\omega$ on the chart $U_l$. We assume the functions $\eta_l$ satisfy the bound $\max_l\norm{\eta_l}_{C^{2k+2}\p{M}}\leq C_\eta$, define the space $WB_k^p\p{M_T}$ analogously to (\ref{WB}) and consider the map
\begin{align*}
&\mathscr{A}:W^{\p{2k+2},p}\p{M_T,\Sym} \longrightarrow WB^p_k\p{M_T}	\\
&\mathscr{A}u= \p{\partial_t u-\trace_g \bnab^2 u,\beta_g(u),\alpha_g\p{\partial_t u},\MC'_g(\partial_t u),\left. u\right|_{t=0}}.
\end{align*}
We will now construct a map in the other direction. Given an element $\p{F,\Phi^1,\Phi^2,\Phi^3,\gamma}\in WB_k^p\p{M_T}$, we proceed as follows:\\

For each $U_l$, we solve the parabolic initial value (boundary) problem (\ref{reallocalsystem}) with input data 
$$\p{h,f,\phi^1,\phi^2,\phi^3,\omega}=\p{\hat{g}_l+\p{1-\hat{1}_l}\delta,\hat{F}_l,\hat{\Phi}^1_l,\hat{\Phi}^2_l,\hat{\Phi}^3_l,\hat{\gamma}_l}$$
on $\Omega_T=\R^n\times [0,T]$ if $U_l$ is an interior chart, and $\Omega_T=\R^n_+\times [0,T]$ if $U_l$ is a boundary chart. In the above, $\hat{1}_l$ denotes the local representation of the characteristic function of $M$, and $\delta$ the Euclidean metric on $\R^n$. We thus obtain a family of solutions $\tb{v_l}_{l=1}^N$, each of which satisfies a parabolic estimate 
\begin{equation}\label{vlestimates}\norm{v_l}_{W^{\p{2k+2},p}\p{\Omega_T}}\leq C\norm{\p{\hat{F}_l,\hat{\Phi}^1_l,\hat{\Phi}^2_l,\hat{\Phi}^3_l,\hat{\gamma}_l}}_{WB_k^p\p{\Omega_T}}\end{equation}
from this family, we define the map 
\begin{align*}
&\mathscr{B}: WB^p_k\p{M_T}\longrightarrow W^{\p{2k+2},p}\p{M_T,\Sym} 	\\
&\mathscr{B}\p{F,\Phi^1,\Phi^2,\Phi^3,\gamma}= \sum_{l=1}^N \eta_l \psi_l^\ast v_l.
\end{align*}
It is clear from (\ref{vlestimates}), that $\mathscr{B}$ defines a bounded operator. Moreover 
\begin{align*} \norm{\p{\partial_t-\trace_g \bnab^2}\sum_{l=1}^N \eta_l \psi_l^\ast v_l-F}_{W^{\p{2k},p}\p{M_T}} & = \norm{g^{-1}\star \p{\bnab^2 \eta_l\star \psi_l^\ast v_l+\bnab \eta_l\star \bnab \psi_l^\ast v_l}}_{W^{\p{2k},p}\p{M_T}}\\ 
& \leq C\p{g}\sum_{l=1}^N\norm{\hat{\eta}_l}_{W^{(2k+2),p}\p{\Omega_T}}\norm{v_l}_{C^{\p{2k+1}}\p{\Omega_T}}\\ & \leq C\p{g,C_\eta}T^{\frac{1}{p}}\norm{\p{F,\Phi^1,\Phi^2,\Phi^3,\gamma}}_{WB_k^p\p{M_T}}
\end{align*}
where we used Sobolev embedding on the norm of $v_l$, and the fact that $\hat{\eta}_l$ is independent of time, which gives
$$\norm{\hat{\eta}_l}_{W^{(2k+2),p}\p{\Omega_T}}=T^{\frac{1}{p}}\norm{\hat{\eta}_l}_{W^{2k+2,p}\p{\Omega}}.$$
Identical estimates can be derived for $\beta_g$ and $\MC_g'$, using the Solonnikov embedding:
\begin{align*}\norm{\beta_g\p{\sum_{l=1}^N\eta_l\psi^\ast_l v_l}-\Phi^1}_{B^{\p{2k+1-\frac{1}{p}},p}\p{\partial M_T}}& =\norm{\sum_{l=1}^N g^{-1}\star \nabla \eta_l\star  \psi_l^\ast v_l}_{B^{\p{2k+1-\frac{1}{p}},p}\p{\partial M_T}}\\ &\leq C\p{g}\sum_{l=1}^N\norm{\hat{\eta}_l}_{W^{(2k+2),p}\p{\Omega_T}}\norm{v_l}_{C^{\p{2k+1}}\p{\Omega_T}}
\end{align*}
\begin{align*}
\norm{\MC_g'\p{\partial_t\sum_{l=1}^N\eta_l\psi^\ast_l v_l}-\Phi^3}_{B^{\p{2k-1-\frac{1}{p}},p}\p{\partial M_T}}& =\norm{\sum_{l=1}^N g^{-1}\star \nabla \eta_l\star  \psi_l^\ast \p{\partial_t v_l}}_{B^{\p{2k-1-\frac{1}{p}},p}\p{\partial M_T}}\\ &\leq C\p{g}\sum_{l=1}^N\norm{\hat{\eta}_l}_{W^{(2k+2),p}\p{\Omega_T}}\norm{v_l}_{C^{\p{2k+1}}\p{\Omega_T}}.
\end{align*}
Combining all these estimates, and using that $\alpha_g\p{\partial_t \sum_{l=1}^N\eta_l\psi^\ast_l v_l}=\Phi^2$, we have 
\begin{equation}\label{opbound}\norm{\mathscr{A}\mathscr{B}-\Id }_{Op}\leq C\p{g,C_\eta}T^{\frac{1}{p}}\end{equation}
where $\norm{\cdot}_{Op}$ denotes the operator norm. Conversely, starting with $u\in W^{(2k+2),p}\p{M_T,\Sym}$, we write $\mathscr{B}\mathscr{A}u=\sum_{l=1}^N \eta_l\psi^\ast_l v_l$, where $v_l$ is a solution of (\ref{reallocalsystem}) with $h=\hat{g}_l+\p{1-\hat{1}_l}\delta$ and input data
\begin{align*}&\p{f,\phi^1,\phi^2,\phi^3,\omega}=\p{\psi_l}_\ast\p{\chi_l\partial_t u-\chi_l\trace_g \bnab^2 u,\chi_l \beta_g\p{u},\chi_l\alpha_{g}\p{\partial_t u},\chi_l\MC_g'\p{\partial_t u},\chi_l\left.u\right|_{t=0}}.\end{align*}
In particular, each $v_l$ satisfies 
\begin{equation}\label{vlulestimate}\norm{v_l}_{W^{(2k+2),p}\p{\Omega_T}}\leq C\p{\hat{g}_l}\norm{\hat{u}_l}_{W^{(2k+2),p}\p{\Omega_T}}.\end{equation}
We obviously have
\begin{equation}\label{BAu-u}\norm{\mathscr{B}\mathscr{A}u-u }_{W^{(2k+2),p}\p{M_T}}=\norm{\sum_{l=1}^N \eta_l\p{\psi^\ast_l v_l-u} }_{W^{(2k+2),p}\p{M_T}}\leq \sum_{l=1}^N\norm{\eta_l\p{\psi^\ast_l v_l-u}}_{W^{(2k+2),p}\p{M_T}} \end{equation}
and, by the definition of the $W^{\p{2k+2},p}$ norm,
$$\norm{\eta_l\p{\psi^\ast_l v_l-u}}_{W^{(2k+2),p}\p{M_T}}=\sum_{m=1}^N\norm{\p{\psi_m}_\ast \p{\chi_m\eta_l\p{\psi^\ast_l v_l-u}}}_{W^{(2k+2),p}\p{\Omega_T}}\leq C\norm{\hat{\eta}_l\p{v_l-\hat{u}_l}}_{W^{(2k+2),p}\p{\Omega_T}}$$
where $C$ is a constant depending on $N$ and the norms of the first order derivative of the transition functions $\psi_m^{-1}\circ \psi_l$. 
%% the trick is towrite $\p{\psi_m}_ast \p{\psi_l^\ast v_l-u}=\p{\psi_m^{-1}\circ \psi_l}^\ast\p{v_l-\p{\psi_l}_\ast u}$
Now set $w_l:=\hat{\eta}_l\p{v_l-\hat{u}_l}$, then by (\ref{vlulestimate}) and the triangle inequality, $w_l$ satisfies 
\begin{equation}\label{triangleinequalityforw}\norm{w_l}_{W^{(2k+2),p}\p{\Omega_T}}\leq C\p{\hat{g}_l}\norm{\hat{u}_l}_{W^{(2k+2),p}\p{\Omega_T}}\end{equation}
Furthermore, since $\eta_l$ kills derivatives of $\chi_l$, we see that $w_l$ solves the system (\ref{reallocalsystem}) with data
\begin{align*} &\p{f,\phi^1,\phi^2,\phi^3,\omega} =\\
& \p{\hat{g}_l^{-1}\star\p{\bnab^2 \hat{\eta}_l\star w_l+\bnab \hat{\eta}_l\star \bnab w_l},\hat{g}_l^{-1}\star \p{\bnab \hat{\eta}_l+\hat{\eta}_l\star \bnab \hat{g}_l}\star w_l,0,\hat{g}_l^{-1}\star \p{\bnab \hat{\eta}_l+\hat{\eta}_l\star \bnab \hat{g}_l}\star \partial_t w_l,0}\end{align*}
from which comes the parabolic estimate of Proposition \ref{upperhalfplane}
\begin{align*}\norm{w_l}_{W^{(2k+2),p}\p{\Omega_T}}\leq & C\p{\hat{g}_l}\left(\norm{\bnab^2 \hat{\eta}_l\star w_l+\bnab \hat{\eta}_l\star \bnab w_l}_{W^{(2k),p}\p{\Omega_T}}+\norm{\bnab \hat{\eta}_l\star w_l}_{B^{\p{2k+1-\frac{1}{p}},p}\p{\partial \Omega_T}} \right.\\ &\qquad \qquad \qquad \quad  \left. +\norm{\bnab \hat{\eta}_l\star \partial_t w_l}_{B^{\p{2k-1-\frac{1}{p}},p}\p{\partial \Omega_T}} \right) \end{align*}
and, by application of the Solonnikov embedding, followed by the Sobolev embedding and use of (\ref{triangleinequalityforw}), we obtain
\begin{align*}\norm{w_l}_{W^{(2k+2),p}\p{\Omega_T}}\leq & C\p{\hat{g}_l}\norm{\hat{\eta}_l}_{W^{(2k+2),p}\p{\Omega_T}}\norm{w_l}_{C^{(2k+1)}\p{\Omega_T}}\leq C\p{\hat{g}_l}\norm{\hat{\eta}_l}_{C^{2k+2}\p{\Omega}}T^{\frac{1}{p}}\norm{\hat{u}_l}_{W^{(2k+2),p}\p{\Omega_T}}\end{align*}
Returning to (\ref{BAu-u}), we can now complete the estimate
\begin{align*}\norm{\mathscr{B}\mathscr{A}u-u}_{W^{(2k+2),p}\p{M_T}} & \leq C\sum_{l=1}^N\norm{w_l}_{W^{(2k+2),p}\p{\Omega_T}}\\ & \leq C\p{\hat{g}_l, C_\eta}T^{\frac{1}{p}}\sum_{l=1}^N\norm{\hat{u}_l}_{W^{(2k+2),p}\p{\Omega_T}}\\ &=	C\p{g, C_\eta}T^{\frac{1}{p}}\norm{u}_{W^{(2k+2),p}\p{M_T}}
\end{align*}
Thus $\mathscr{B}\mathscr{A}-\Id$ also satisfies the operator bound (\ref{opbound}). For every $T$ smaller than a critical size, depending only on the initial metric $g$, we can ensure that the bound itself is strictly less that $1$. With that follows invertibility of $\mathscr{A}\mathscr{B}$ and $\mathscr{B}\mathscr{A}$, from which we construct a left and right inverse of $\mathscr{A}$:
$$\p{\mathscr{B}\mathscr{A}}^{-1}\mathscr{B}\qquad \qquad \text{resp.} \qquad \qquad \mathscr{B}\p{\mathscr{A}\mathscr{B}}^{-1}$$
Thus $\mathscr{A}$ must in fact have a bounded inverse $\mathscr{A}^{-1}$, which to any element $\p{F,\Phi^1,\Phi^2,\Phi^3,g}\in WB_k^p\p{M_T}$ (note that the last entry is now fixed to the initial metric $g$), associates a unique solution $u\in W^{(2k+2),p}\p{M_T,\Sym}$ of (\ref{initialIBV}), satisfying 
$$\norm{u}_{W^{(2k+2),p}\p{M_T}}\leq \norm{\mathscr{A}^{-1}}_{Op}\norm{\p{F,\Phi^1,\Phi^2,\Phi^3,g}}_{WB_k^p\p{M_T}}$$
By construction, the operator norm of $\mathscr{A}^{-1}$ depends on $g$ and $C_\eta$, and it is independent of $T$, as long as the interval is short enough.
\end{proof}
Note that we refrain from specifying $\overline{g}$ as a dependency in any constant, since we regard this metric as fixed and immutable.
We follow \cite{PG} in defining the following complete space of metrics, in the vicinity of an arbitrary initial metric.

\begin{lemma}\label{damnedestimates} Let $p>n$ and define the space
$$M_K^T(g):=\left\{w\in W^{(4),p}\p{M_T,\Sym}\; ;\; w|_{t=0}=g,\; \partial_t w|_{t=0}=-2\Ric_g ,\; \norm{w-g}_{W^{(4),p}}\leq K\right\}$$
If $T>0$ is small enough, there exist constants, depending on $g$ and $K$ such that for $w\in M_K^T(g)$
$$\norm{w}_{W^{(2),p}\p{M_T}}\leq C(g,K)T^{\frac{1}{p}}$$	
$$\norm{w-g}_{C^{(2)}\p{ M_T}}\leq C(g,K)T^\gamma$$
where $\gamma=\frac{p-n}{2p}$. If $w_1,w_2\in M_K^T(g)$, then
$$\norm{w_1-w_2}_{C^{(2)}\p{ M_T}}\leq CT^\gamma \norm{w_1-w_2}_{W^{(4),p}\p{M_T}}$$
\end{lemma}
\begin{proof} Let $w\in M_K^T(g)$ and let $\hat{w}_l$ be its local representation in the chart $U_l$ (See Definition \ref{fctspacesonM}). Set $v=\partial_t^m D^{\alpha} \hat{w}_l$ for $2m+|\alpha|\leq 2$, then,  by Young's inequality,
\begin{align*}\abs{v(x,t)}^p &\leq p\int_0^T\abs{\partial_t v(x,s)}\abs{v(x,s)}^{p-1}\; ds+\abs{v(x,0)}^p\\ &\leq \int_0^T\abs{\partial_t v(x,s)}^p\; ds+(p-1)\int_0^T\abs{v(x,s)}^p\; ds +\abs{v(x,0)}^p.
	\end{align*}
In particular, for $T<\frac{1}{p}$ 
$$\sup_t\abs{v(x,t)}^p< p\p{\int_0^T\abs{\partial_t v(x,s)}^p\; ds+\abs{v(x,0)}^p}.$$
By the initial conditions imposed on the space $M_K^T(g)$
\begin{align*}\norm{v}_{L^p\p{\Omega_T}}^p &\leq T\int_{0}^T\sup_t \abs{v(x,t)}^p\; dx\leq pT\p{\int_{\Omega_T}\abs{\partial_t v(x,t)}^p\; dx\; dt+\int_\Omega \abs{v(x,0)}^p\; dx}\\ &\leq CT\p{\norm{\hat{w}_l}_{W^{(4),p}\p{\Omega_T}}^p+\norm{\hat{g}_l}_{W^{2,p}\p{\Omega}}}\end{align*}
and the first inequality follows directly. As for the second part, do as follows:
\begin{equation}\label{hatw-hatg}\norm{\hat{w}_l-\hat{g}_l}_{C^{(2)}\p{\Omega_T}}\leq \norm{\hat{w}_l-\hat{g}_l+2t\Ric_{\hat{g}_l}}_{C^{(2)}\p{\Omega_T}}+CT\norm{\hat{g}_l}_{C^4\p{\Omega}}\end{equation}
Then $v:=\hat{w}_l-\hat{g}_l+2t\Ric_{\hat{g}_l}$ satisfies $\partial_t^mv|_{t=0}=0$ for $m=0,1$, which means we can use \cite[Lemma 4.2]{Sol1}:
\begin{align*}\norm{v}_{C^{(2)}\p{\Omega_T}}\leq &CT^{\frac{\varepsilon}2}\left\langle v\right\rangle_{\frac{2+\varepsilon}{2}, t}
\end{align*}
where 
\begin{align*}\left\langle v\right\rangle_{\frac{2+\varepsilon}{2},t}=& \sup_{x,t,s}\frac{\abs{\partial_t v(x,t)-\partial_t v(x,s)}}{\abs{t-s}^{\varepsilon/2}}+\sum_{|\alpha|=1}^2\sup_{x,t,s}\frac{\abs{D^\alpha v(x,t)-D^\alpha v(x,s)}}{\abs{t-s}^{(2+\varepsilon-\abs{\alpha})/2}}\\ =& \sup_{x,t,s} \frac{\abs{\partial_t \hat{w}_l(x,t)-\partial_t \hat{w}_l(x,s)}}{\abs{t-s}^{\varepsilon/2}}+ \sum_{|\alpha|=1}^2\sup_{x,t,s}\frac{\abs{D^\alpha \hat{w}_l(x,t)-D^\alpha \hat{w}_l(x,s)+2(t-s)D^\alpha\Ric_{\hat{g}_l}}}{\abs{t-s}^{(2+\varepsilon-\abs{\alpha})/2}} \\ \leq &\norm{\hat{w}_l}_{C^{(2+\varepsilon)}\p{\Omega_T}}+2T^{\frac{1-\varepsilon}{2}}\norm{\hat{g}_l}_{C^4\p{\Omega}}.
	\end{align*}
Consequently, plugging the above into (\ref{hatw-hatg}), we get
$$\norm{w-g}_{C^{(2)}\p{M_T}}\leq C_1T^{\frac{\varepsilon}2}\norm{w}_{C^{(2+\varepsilon)}\p{M_T}}+C_2T^{\frac{1}2}\norm{g}_{C^4\p{M}}$$
for any given $0<\varepsilon<1$. By choosing $\varepsilon=1-\frac{n}{p}$, we obtain the desired inequality by Sobolev embedding: $\norm{w}_{C^{(2+\varepsilon)}}\leq C\norm{w}_{W^{(4),p}}$. A similar, but simpler, calculation shows the inequality for $w_1-w_2$, since they satisfy $\left.\partial_t^m(w_1-w_2)\right|_{t=0}=0$ for $m=0,1$, without the need to add any terms.
\end{proof}

An application of Solonnikov embedding, Theorem \ref{SolEmb}, that we shall revisit several times in the coming pages, is when $u\in M_K^T(g)$:
$$\norm{\partial_t u}_{B^{\p{2-\frac{1}{p}},p}\p{\partial M_T}}=\norm{\partial_t \p{u-g}}_{B^{\p{2-\frac{1}{p}},p}\p{\partial M_T}}\leq C\norm{u-g}_{W^{(4),p}\p{M_T}}.$$

Assuming $K>0$ is small enough, such that every $w\in M_K^T(g)$ is a metric, we can apply the following constructions to the entirety of this space.

\begin{definition}\label{Rdtfield} Recall that the difference of two Christoffel symbols is a tensor:
$$T\p{g_1,g_2}_{ij}^k:=\Gamma\p{g_1}_{ij}^k-\Gamma\p{g_2}_{ij}^k.$$
When $g_2=\overline{g}$ is the background metric, we write this in local coordinates as 
\begin{equation}\label{Tlocalcoord}T(g,\overline{g})_{ij}^k=\frac{1}{2}g^{kl}\p{\bnab_i g_{jl}+\bnab_j g_{il}-\bnab_l g_{ij}}\end{equation}
Let $\tilde{g}_t$ be a 1-parameter family of metrics with $\tilde{g}_0=g$, and let $w$ be a non-degenerate symmetric 2-tensor. Define the fields 
\begin{itemize}
\item The DeTurck tensor/vector field 
$$\xi_t(w)_i=w_{ij}w^{kl}T\p{\tilde{g}_t,w}_{kl}^j\qquad \qquad \xi_t^{\#}(w)^{ j}=w^{kl}T\p{\tilde{g}_t,w}_{kl}^j.$$
If $g_s=g+sh$: 
$$\frac{d}{ds}\xi_t\p{g_s}_{i}=\beta_{g_s}(h)_i+\p{h_{ij}g_s^{kl}-g^s_{ij}h^{kl}}T\p{\tilde{g}_t,g_s}_{kl}^j$$
where $\beta_g\p{h}=\delta_g h+\frac{1}{2}d\trace_g h$ is the Bianchi operator. We will always assume that $\tilde{g}_t$ is chosen to satisfy $\left.\partial_t\tilde{g}_t\right|_{t=0}=0$ and  $\norm{\tilde{g}_t-g}_{W^{(4),p}\p{M_T}}\leq C\p{g,\tilde{g}_t}T^{\frac{1}{p}}$, for all reasonably small $T$.
\item The remainder tensor
 \begin{align*}\mathcal{R}\p{w,\overline{\nabla}w,\tilde{g},\bnab \tilde{g},\bnab^2\tilde{g}}=& w^{lm}w_{pj}\p{2T^p_{im}T^{i}_{lk}+2\tilde{T}^p_{im}T^{i}_{lk}-\tilde{T}^{i}_{lm}T_{ik}^p-\overline{R}_{klm}^p}\\ &+w^{lm}w_{pk}\p{2T^p_{im}T^{i}_{lj}+2\tilde{T}^p_{im}T^{i}_{lj}-\tilde{T}^{i}_{lm}T_{ij}^p-\overline{R}_{jlm}^p}\\ &-2w^{im}w_{pl}T_{ik}^pT_{jm}^l-w^{lm}\p{w_{ik}\bnab_j +w_{ij}\bnab_k}\tilde{T}_{lm}^{i}\end{align*}
 where $T:=T(w,\overline{g})$ and $\tilde{T}:=T\p{\tilde{g},\overline{g}}$.
\end{itemize}
\end{definition}

At the opening of this section, we proved an existence and uniqueness result for an IBP with stationary boundary conditions (Theorem \ref{initialRDT}). To apply this to the Ricci-deTurck flow equation,
\begin{equation}\label{realRdT}
	\partial_t w=-2\Ric_w-2\delta_w^*\xi_t(w)
\end{equation}
we will show that it is possible to write it as a heat-like equation, specifically the one we had as the main equation of (\ref{initialIBV}).

\begin{proposition}\label{RDTtoHeat} The Ricci-deTurck flow equation (\ref{realRdT}) can be rewritten as 
\begin{align}\label{heatformRdT}
	\partial_t w-\trace_w\bnab^2 w=\mathcal{R}\p{w,\bnab w,\tilde{g},\bnab\tilde{g},\bnab^2\tilde{g}}\end{align}
\end{proposition}
\begin{proof} We begin by splitting the deTurck field
$$\xi_t(w)_i=w_{ij}w^{kl}\tilde{T}_{kl}^j-w_{ij}w^{kl}T_{kl}^j=:V(w)_i-W(w)_i.$$
From \cite[equation 2.49]{CLN}, we have
\begin{equation}\label{Rw1}R(w)_{ijk}^l=\overline{R}_{ijk}^l+\bnab_i T_{jk}^l-\bnab_j T_{ik}^l+T_{jk}^pT_{ip}^l-T_{ik}^pT_{jp}^l\end{equation}
Using the local coordinates of (\ref{Tlocalcoord}), we expand the second and third terms:
\begin{align*}\bnab_i T_{jk}^l-\bnab_j T_{ik}^l =&\frac{1}{2}w^{lm}\p{\bnab_i\bnab_k w_{jm}-\bnab_i\bnab_m w_{jk}-\bnab_j\bnab_k w_{im}+\bnab_j\bnab_m w_{ik}}+\frac{1}{2}w^{lm}\p{w_{mp}\overline{R}^p_{jik}+w_{kp}\overline{R}_{jim}^p}\\ &+w^{lm}\p{T_{ik}^p\bnab_j-T^p_{jk}\bnab_i}w_{mp}.
\end{align*}
Summing over $i$ and $l$, (\ref{Rw1}) becomes 
\begin{align*}-\overline{\Ric}_{jk}-&\frac{1}{2}w^{im}\bnab_i\bnab_m w_{jk}-\frac{1}{2}w^{im}\bnab_j\bnab_k w_{im}+\frac{1}{2}w^{im}\p{\bnab_j\bnab_i w_{mk}+\bnab_k\bnab_i w_{mj}} + w^{im}\p{T_{ik}^p\bnab_j-T^p_{jk}\bnab_i}w_{mp}\\+& \frac{1}{2}w^{im}\p{w_{pj}\overline{R}^p_{kim}+w_{pk}\overline{R}^p_{jim}}
\end{align*}
meanwhile
\begin{align*}2\delta^*_w W(w)_{jk}=& w^{im}\p{w_{lk}\bnab_j+w_{lj}\bnab_k}T_{im}^{l}+w^{im}T_{im}^{l}\bnab_l w_{jk}+w^{ip}w^{mq}T_{im}^{l}\p{w_{lk}\bnab_j+w_{lj}\bnab_k}w_{pq}\\
= &w^{im}\p{\bnab_j\bnab_i w_{mk}+\bnab_k\bnab_i w_{mj}}-\frac{1}{2}w^{im}\p{\bnab_j\bnab_k+\bnab_k\bnab_j}w_{im}-2w^{im}w_{lp}T_{jk}^pT_{im}^{l}\\ &+w^{ip}w^{mq}T_{im}^{l}\p{w_{lk}\bnab_j+w_{lj}\bnab_k}w_{pq}.
\end{align*}
In combination with the rest of (\ref{Rw1});
\begin{align*}-2\Ric(w)_{jk}+2\delta^*_w W(w)_{jk}=&w^{im}\bnab_i\bnab_m w_{jk}-w^{im}\p{w_{pj}\overline{R}_{kim}^p+w_{pk}\overline{R}_{jim}^p}-2w^{im}w_{pl}T_{ik}^pT_{jm}^l\\&+w^{ip}w^{mq}T_{im}^l\p{w_{lk}\bnab_j+w_{lj}\bnab_k}w_{pq}\\
= & w^{im}\bnab_i\bnab_m w_{jk}-w^{im}\p{w_{pj}\overline{R}_{kim}^p+w_{pk}\overline{R}_{jim}^p}\\ &+ 2w^{ip}w_{lk}T^{l}_{im}T_{pj}^m+2w^{ip}w_{lj}T_{im}^lT_{pk}^m-2w^{im}w_{pl}T_{ik}^pT_{jm}^l.
\end{align*}
Similarly,
	\begin{align*}
	2\delta_{w}^*V(w)_{jk}=& 2\delta_{\overline{g}}^*V(w)-2w^{lm}\tilde{T}_{lm}^{i}w_{ip}T^{p}_{jk}\\ =& w^{lm}\tilde{T}_{lm}^{i}\bnab_i w_{jk}+\tilde{T}_{lm}^{i}\p{w_{ik}\bnab_j +w_{ij}\bnab_k}w^{lm}+w^{lm}\p{w_{ik}\bnab_j +w_{ij}\bnab_k}\tilde{T}_{lm}^{i}.\end{align*}
Thus, finally
\begin{align*}-2\Ric(w)_{jk}+2\delta^*_w W(w)_{jk}-2\delta^*_w V(w)_{jk}=& w^{im}\bnab_i\bnab_m w_{jk}-2w^{im}w_{pl}T_{ik}^pT_{jm}^l-w^{lm}\p{w_{ik}\bnab_j +w_{ij}\bnab_k}\tilde{T}_{lm}^{i}\\& +w^{lm}w_{pj}\p{2T^p_{im}T^{i}_{lk}+2\tilde{T}^p_{im}T^{i}_{lk}-\tilde{T}^{i}_{lm}T_{ik}^p-\overline{R}_{klm}^p}\\ &+w^{lm}w_{pk}\p{2T^p_{im}T^{i}_{lj}+2\tilde{T}^p_{im}T^{i}_{lj}-\tilde{T}^{i}_{lm}T_{ij}^p-\overline{R}_{jlm}^p}.\end{align*}
\end{proof}

The heat-form Ricci-deTurck equation (\ref{heatformRdT}) is, of course, not directly insertable in the IBP (\ref{initialIBV}), as both sides depend on $w$. It does, however, allow us to construct the aforementioned contraction argument, in the style of \cite{PG}. There is just one more hurdle to overcome, before we can find a solution to the Ricci-deTurck IBP: That the "I" and the "B" are compatible, in a way that allows the system to start at time $t=0$. We will say an initial metric $g$ is \textit{RdT-compatible} if, at the boundary,
\begin{align*}
\Ric_g^T-\varphi_g\p{\Ric_g}g^T=0\qquad \text{and}\qquad \MC_g'\p{\Ric_g}+\frac{1}{2}\varphi_g\p{\Ric_g}\MC_g=0	
\end{align*}
This corresponds to zeroth order compatibility conditions for the Ricci-deTurck flow IBP (and indeed the corresponding Ricci flow IBP). It is, however, not enough to provide compatibility in the Ricci case, which will require first order compatibility as well; cf. Theorem \ref{Ricciflowintro}.  \\

We now come to the main existence theorem regarding the Ricci-deTurck flow:

\begin{theorem}[Short-term existence of Ricci-deTurck flow] \label{STE} Let $g$ be an RdT-compatible smooth Riemannian metric on $M$ and let $p>n$. For any $K>0$, there exists a $T>0$ such that the IBP
\begin{equation}\label{STERdT}\begin{cases}
\partial_t u=-2\Ric_u-2\delta_u^*\xi_t(u) & \text{in  }M_T\\
\begin{rcases}
\xi_t(u)&=0\\
\MC_u'\p{\partial_t u}+\frac{\varphi_u\p{\partial_t u}}{2}\MC_u&=0\\
\p{\partial_t u}^T-\varphi_u\p{\partial_t u}u^T &=0	
\end{rcases} & \text{on  }\partial M_T\\
\left.u\right|_{t=0}=g
\end{cases}
\end{equation}
has a unique solution $u\in W^{(4),p}\p{M_T}$, satisfying $\norm{u-g}_{W^{(4),p}(M_T)}\leq K$.\end{theorem}

\begin{proof}
Consider the the space $M_K^T(g)$, defined in Lemma \ref{damnedestimates}.  We may assume that $T$ is small enough such that every element $w$ of $ M_K^T(g)$ is a metric and $\norm{w^{-1}}_{C^{(2)}\p{M_T}}\leq C(g,K)$. Utilizing the remainder tensor $\mathcal{R}$, defined in Definition \ref{Rdtfield}, we set
$$F_w:=\mathcal{R}\p{w,\overline{\nabla}w,\tilde{g},\bnab \tilde{g},\bnab^2\tilde{g}}+\trace_w\overline{\nabla}^2 w-\trace_g\overline{\nabla}^2 w$$
$$\Phi_w^1:=\beta_g(w)-\xi_t(w)$$
$$\Phi_w^2:=\varphi_w\p{\partial_t w}w^T-\varphi_g\p{\partial_t w}g^T$$
$$\Phi_w^3:=\MC'_g\p{\partial_t w}-\MC'_w\p{\partial_t w}-\frac{\varphi_w\p{\partial_t w}}{2}\MC_w.$$
Note that $\p{F_w,\Phi^1_w,\Phi^2_w,\Phi^3_w,g}\in WB_1^p\p{M_T}$, so Theorem \ref{initialRDT} provides unique solvability of the IBP
\begin{equation}\label{IBVS}\begin{cases}
\partial_t u-\trace_g\overline{\nabla}^2u \qquad \quad \;\;\; =F_w & \text{in  }M_T\\
\begin{rcases}
\beta_g(u)&=\Phi_w^1\\
\p{\partial_t u}^T-\varphi_g\p{\partial_t u}g^T &=\Phi_w^2\\
\MC_g'\p{\partial_t u}&=\Phi_w^3
\end{rcases} & \text{on  }\partial M_T\\
\left.u\right|_{t=0}\qquad \qquad \qquad \quad\:\: =g.
\end{cases}
\end{equation}
The solution is of class $W^{(4),p}\p{M_T}$, and it qualifies for the parabolic estimate (\ref{parabolicest}) in this norm. Thus, we have a well-defined solution-map
$$S:M_K^T(g)\to W^{(4),p}\p{M_T}$$
that we will prove is a contraction for small $T$. If we allow the background metric $\overline{g}$ to be subsumed by the $\star$-contractions, we may write 
\begin{align}\label{Rtensor}\mathcal{R}\p{w,\overline{\nabla}w,\tilde{g},\bnab \tilde{g},\bnab^2\tilde{g}}=w^{-1}\star w\star \p{\overline{R}+\bnab\tilde{T}}+w^{-2}\star \p{\bnab w}^2+w^{-2}\star \tilde{g}^{-1}\star w\star \bnab w\star \bnab \tilde{g}.\end{align}
We will now give an estimate for the $W^{(2),p}\p{M_T}$ norm of $\mathcal{R}$. To allow us to focus on the essential terms, we will note that we will make repeated use of the bounds $\norm{w^{-1}}_{C^{(2)}\p{M_T}}\leq C\p{g,K}$ and $\norm{\tilde{g}^{-1}}_{C^{(2)}\p{M_T}}\leq C\p{\tilde{g}}$. Thus, for example,
$$\norm{w^{-1}\star w\star \p{\overline{R}+\bnab \tilde{T}}}_{W^{\p{2},p}}\leq C(g,K)\norm{w}_{W^{(2),p}}+C(g,K)\norm{w\star \bnab\tilde{T}}_{W^{(2),p}}.$$
By expanding the second order term 
$$\bnab\tilde{T}=\tilde{g}^{-2}\star \p{\bnab \tilde{g}}^2+\tilde{g}^{-1}\star \bnab^2\tilde{g}$$
we find that 
$$\norm{w\star \bnab\tilde{T}}_{W^{(2),p}}\leq C\p{\tilde{g}}\p{\norm{w\star \p{\bnab \tilde{g}}^2}_{W^{(2),p}}+\norm{w\star \bnab^2\tilde{g}}_{W^{(2),p}}}$$
and, by expanding the $W^{\p{2},p}$ norm and estimating term by term, 
$$\norm{w\star \p{\bnab \tilde{g}}^2}_{W^{(2),p}}\leq \norm{\tilde{g}}_{C^{(2)}}^2\norm{w}_{W^{(2),p}}+\norm{w}_{C^{(0)}}\norm{\tilde{g}}_{C^{(1)}}\norm{\tilde{g}}_{W^{(4),p}}\leq C(\tilde{g})\norm{w}_{W^{(2),p}}\leq C\p{g,\tilde{g}}T^{\frac{1}{p}}$$
$$\norm{w\star \bnab^2\tilde{g}}_{W^{(2),p}}\leq \norm{w}_{C^{(2)}}\norm{\tilde{g}-g}_{W^{(4),p}}+\norm{g}_{C^{4}}\norm{w}_{W^{(2),p}}\leq C\p{g,\tilde{g}}T^{\frac{1}{p}}$$
 where we applied Lemma \ref{damnedestimates}, our choice of $\tilde{g}$ (See Def. \ref{Rdtfield}) and the Sobolev embedding: $\norm{w}_{C^{(2)}}\leq C\norm{w}_{W^{(4),p}}\leq C\p{g,K}$. We use a similar tactic of splitting the norm on the second term of (\ref{Rtensor}), as well as an instance of the Sobolev embedding ($\norm{w}_{C^{(1)}}\leq \norm{w}_{W^{(2),p}}$):
 $$\norm{\p{\bnab w}^2}_{W^{(2),p}}\leq 2\norm{w}_{C^{(1)}}\norm{w}_{W^{(4),p}}+\norm{w}_{C^{(2)}}\norm{w}_{W^{(2),p}}\leq C\p{g,K}\norm{w}_{W^{\p{2},p}}$$
and lastly
\begin{align*}\norm{w\star \bnab w\star \bnab \tilde{g}}_{W^{(2),p}} &\leq \norm{w}_{C^{(1)}}\norm{\tilde{g}}_{C^{(2)}}\norm{w}_{W^{(2),p}}+\p{\norm{\tilde{g}}_{C^{(1)}}\norm{w}_{W^{(4),p}}+\norm{w}_{C^{(1)}}\norm{\tilde{g}}_{W^{(4),p}}}\norm{w}_{C^{(0)}}\\ &\leq C\p{g,\tilde{g},K}\norm{w}_{W^{(2),p}}\leq C\p{g,\tilde{g},K}T^{\frac{1}{p}}. \end{align*}
As for the remainder of $F_w$, we will use $w^{ij}-g^{ij}=g^{ki}w^{lj}\p{g_{kl}-w_{kl}}$:
\begin{align*}\norm{\trace_w\overline{\nabla}^2w-\trace_g\overline{\nabla}^2 w}_{W^{(2),p}}=\norm{(w^{-1}-g^{-1})\star \overline{\nabla}^2w}_{W^{(2),p}}\leq C\p{g,K}\norm{w-g}_{C^{(2)}}\norm{w}_{W^{(4),p}}\leq C(g,K)T^{\gamma}\end{align*}
where $\gamma$ comes from part $(ii)$ of Lemma \ref{damnedestimates}. We now have
$$\norm{F_w}_{W^{(2),p}}\leq C(g,\tilde{g},K)T^{\gamma_p}$$
where $\gamma_p=\min\left\{\frac{1}{p},\gamma=\frac{p-n}{2p}\right\}$. By writing $T\p{g_1,g_2}=T\p{g_1,\overline{g}}-T\p{g_2,\overline{g}}$, we obtain by (\ref{Tlocalcoord}) the expression
\begin{align*}T\p{g_1,g_2}= g_2^{-1}\star \bnab \p{g_1-g_2}+\p{g_1^{-1}-g_2^{-1}}\star \bnab g_1.\end{align*}
In the case where $g_1=\tilde{g}$ and $g_2\in M_K^T\p{g}$, the Solonnikov embedding (Theorem \ref{SolEmb}) implies
$$\norm{T\p{\tilde{g},g_2}}_{B^{\p{3-\frac{1}{p}},p}}\leq C\p{g,\tilde{g},K}\norm{\tilde{g}-g_2}_{W^{\p{4},p}}.$$
If we set $h:=w-g$ and $g_s:=g+sh$, then $g_s\in M_K^T\p{g}$ for all $0\leq s\leq 1$, and we can write 
$$\xi_t\p{w}_m=\int_{0}^1\frac{d}{ds}\xi_t(g_s)_m\; ds +\xi_t(g)_m=\int_0^1\beta_{g_s}(h)_m\; ds+\int_0^1 A(s,t)_m\; ds$$
where 
\begin{align*}A(s,t)_m:=&\p{h_{lm}g_s^{ik}-g^s_{lm}g_s^{ij}g_s^{nk}h_{jn}}T\p{\tilde{g},g_s}+\xi_t(g)_m \\ =&\p{g_s^{-1}-g_s^{-1}\star g_s^{-1}\star g_s}\star h\star T\p{\tilde{g},g_s}+g^{-1}\star g\star T\p{\tilde{g},g}.\end{align*}
Recalling that $\tilde{g}$ is chosen to satisfy $\norm{\tilde{g}-g}_{W^{(4),p}(M_T)}\leq C(g,\tilde{g})T^{\frac{1}{p}}$, we obtain by the Solonnikov embedding
$$\norm{A}_{B^{\p{3-\frac{1}{p}},p}}\leq C\p{g,\tilde{g}}\p{\norm{h}_{C^{(2)}}\norm{\tilde{g}-g_s}_{W^{(4),p}}+\norm{\tilde{g}-g}_{W^{(4),p}}}\leq C\p{g,\tilde{g},K}T^{\gamma_p}.$$
The Bianchi operator has the compact form 
$$\beta_g(h)=g^{-1}\star \bnab h+g^{-2}\star \bnab g\star h$$
so the rest of $\Phi^1_w$ can be written as 
$$B(s):=\beta_g(h)-\beta_{g_s}(h)=\p{g^{-1}-g_s^{-1}}\star \bnab h+ \p{g^{-2}-g_s^{-2}}\star \bnab g\star h+g_s^{-2}\star \bnab\p{g-g_s}\star h$$
so by the Solonnikov embedding
$$\norm{B}_{B^{\p{3-\frac{1}{p}},p}}\leq C\p{g,K}\norm{h}_{C^{(2)}}\norm{h}_{W^{(4),p}}\leq C(g,K)T^{\gamma}.$$
We can now give the estimate
$$\norm{\Phi_w^1}_{B^{\p{3-\frac{1}{p}},p}}=\norm{\int_0^1 B-A\; ds}_{B^{\p{3-\frac{1}{p}},p}}\leq C(g,\tilde{g},K)T^{\gamma_p}.$$
The second boundary operator poses no great problem: If $k:=\partial_t h=\partial_t w$, then
$$\Phi^2_w=\p{w^{-1}\star w-g^{-1}\star g}\star k$$
which means 
$$\norm{\Phi^2_w}_{B^{\p{2-\frac{1}p},p}}\leq C(g,K)\norm{w-g}_{C^{(2)}}\norm{k}_{B^{\p{2-\frac{1}p},p}}\leq C(g,K)T^\gamma\norm{w-g}_{W^{(4),p}}\leq C(g,K)T^\gamma$$
where $\norm{k}_{B^{\p{2-\frac{1}{p}}}}\leq C\norm{w-g}_{W^{(4),p}}$ follows directly from the Solonnikov embedding. We will now estimate the third boundary operator. To that end, we denote by $\nu$ and $\dnu$ local representations of the unit normals of $g$ and $w$, respectively. The linearization of the mean curvature can be written most compactly as 
$$\MC_g'(k)=\p{g^{-2}+g^{-1}}\star \bnab \p{g\star \nu \star k}.$$
Thus
\begin{align*}\label{mean-mean}\MC_g'(k)-\MC_w'(k)=& \p{g^{-2}-w^{-2}+g^{-1}-w^{-1}}\star \bnab \p{g\star \nu\star k}+\p{w^{-2}+w^{-1}}\star \bnab\p{\p{g\star \nu-w\star \dnu}\star k}.\end{align*}
Using the bound $\abs{\nu-\dnu}_{\overline{g}}\leq C\abs{g-w}_{\overline{g}}$ and Solonnikov embedding, we obtain
\begin{align*}
\norm{\MC_g'(k)-\MC_w'(k)}_{B^{\p{1-\frac{1}{p}},p}}\leq & C(g,K)\norm{g-w}_{C^{(0)}}\norm{g\star \nu \star k}_{W^{(2),p}}+C(g,K)\norm{\p{g\star \nu-w\star \dnu}\star k}_{W^{(2),p}}\\ \leq & C(g,K)\norm{g-w}_{C^{(2)}}\norm{w}_{W^{(4),p}}\\ \leq & C(g,K)T^\gamma
\end{align*} 
by the usual application of Lemma \ref{damnedestimates}. Lastly, 
$$\varphi_w\p{k}\MC_w=w^{-2}\star k\star \bnab \p{\dnu \star w}$$
from which we get
$$\norm{\varphi_w\p{k}\MC_w}_{B^{\p{1-\frac{1}{p}},p}}\leq C(g,K)\norm{k}_{C^{\p{0}}}\norm{\dnu \star w}_{W^{\p{2},p}}\leq C(g,K)\norm{w-g}_{C^{(2)}}\leq C(g,K)T^\gamma.$$
The metric $u:=S(w)-g$ solves an IBP almost identical to (\ref{IBVS}), the only differences being the initial value $\left.u\right|_{t=0}=0$ and the interior identity 
$$\partial_t u-\trace_g\overline{\nabla}^2u=F_w+\trace_g\overline{\nabla}^2g.$$
The corresponding parabolic estimate, and the work we did before, shows
\begin{align*}\norm{S(w)-g}_{W^{(4),p}}&\leq C\norm{F_w+\trace_g\overline{\nabla}^2g}_{W^{(2),p}}+C\norm{\Phi_w^1}_{B^{\p{3-\frac{1}p},p}}+C\norm{\Phi^2_w}_{B^{\p{2-\frac{1}{p}},p}}+C\norm{\Phi^3_w}_{B^{\p{1-\frac{1}{p}},p}}\\ &\leq C(g,\tilde{g},K)T^{\gamma_p}\end{align*}
for $T$ small enough. This shows we can choose $T$ sufficiently small, such that $S$ is a map from $M_K^T(g)$ to itself.\\

\noindent To do a contraction argument, we take $w_1,w_2\in M_K^T\p{g}$ and set out to estimate $S(w_1)-S(w_2)$. Starting with the remainder tensor,
\begin{align*}
\mathcal{R}&\p{w_1,\bnab w_1,\tilde{g},\bnab \tilde{g},\bnab^2 \tilde{g}}-\mathcal{R}\p{w_2,\bnab w_2,\tilde{g},\bnab \tilde{g},\bnab^2 \tilde{g}}= \p{w_1^{-1}\star w_1-w_2^{-1}\star w_2}\star\p{ \overline{R}+\bnab \tilde{T}}\\ &+w_1^{-2}\star \p{\bnab w_1}^2-w_2^{-2}\star \p{\bnab w_2}^2+\p{w_1^{-2}\star w_1\star \bnab w_1-w_2^{-2}\star w_2\star \bnab w_2}\star \tilde{g}^{-1}\star \bnab \tilde{g}.
\end{align*}
In much the same way as we estimated $\mathcal{R}$ earlier:
\begin{align*}\norm{\p{w_1^{-1}\star w_1-w_2^{-1}\star w_2}\star\p{ \overline{R}+\bnab \tilde{T}}}_{W^{(2),p}}& \leq C\p{g,K}\p{1+\norm{\tilde{g}}_{W^{(4),p}}}\norm{w_1-w_2}_{C^{(2)}}\\ &\leq C\p{g,\tilde{g},K}T^\gamma \norm{w_1-w_2}_{W^{(4),p}},\end{align*}
\begin{align*} &\norm{w_1^{-2}\star \p{\bnab w_1}^2-w_2^{-1}\star \p{\bnab w_2}^{2}}_{W^{(2),p}}\leq \norm{\p{w_1^{-2}-w_2^{-2}}\star \p{\bnab w_1}^2}_{W^{(2),p}}+C\p{g,K}\norm{\p{\bnab w_1}^2-\p{\bnab w_2}^2}_{W^{(2),p}}\\ &\qquad\qquad  \qquad \leq C\p{g,K}\norm{w_1-w_2}_{C^{(2)}}+C\p{g,K}\p{\norm{w_1+w_2}_{W^{(2),p}}\norm{w_1-w_2}_{W^{(4),p}}+\norm{w_1-w_2}_{W^{(2),p}}}\\ &\qquad\qquad \qquad \leq C\p{g,K}T^{\gamma_p}\norm{w_1-w_2}_{W^{(4),p}}
\end{align*}
and 
\begin{align*} & \norm{\p{w_1^{-2}\star w_1\star \bnab w_1-w_2^{-2}\star w_2\star \bnab w_2}\star \tilde{g}^{-1}\star \bnab \tilde{g}}_{W^{(2),p}}\\& \qquad \leq C\p{g,\tilde{g},K}\p{\norm{\p{w_1-w_2}\star \bnab w_1\star \bnab \tilde{g}}_{W^{(2),p}}+\norm{w_2\star \bnab \p{w_1-w_2}\star \bnab \tilde{g}}_{W^{(2),p}}}\\ &\qquad \leq  C\p{g,\tilde{g},K}\p{\norm{\tilde{g}}_{C^{(1)}}\p{\norm{w_1}_{W^{(4),p}}+\norm{w_2}_{W^{(4),p}}}+\p{\norm{w_1}_{C^{(1)}}+\norm{w_2}_{C^{(1)}}}\norm{\tilde{g}}_{W^{(4),p}}}\norm{w_1-w_2}_{W^{(2),p}}\\ &\quad \qquad  +C\p{g,\tilde{g},K}\norm{w_2}_{C^{(0)}}\norm{\tilde{g}}_{W^{(2),p}}\norm{w_1-w_2}_{W^{(4),p}}\\ &\qquad \leq C\p{g,\tilde{g},K}T^{\frac{1}{p}}\norm{w_1-w_2}_{W^{(4),p}}.
\end{align*}
By which we may conclude
\begin{align*}\norm{\mathcal{R}\p{w_1,\bnab w_1,\tilde{g},\bnab \tilde{g},\bnab^2 \tilde{g}}-\mathcal{R}\p{w_2,\bnab w_2,\tilde{g},\bnab \tilde{g},\bnab^2 \tilde{g}}}_{W^{\p{2},p}}\leq C\p{g,\tilde{g},K}T^{\gamma_p}\norm{w_1-w_2}_{W^{(4),p}}.  \end{align*}
As for the rest of $F$, we have 
$$\trace_{w_1}\bnab^2w_1-\trace_{w_2}\bnab^2w_2-\trace_g\bnab^2w_1+\trace_g\bnab^2w_2=\p{w_1^{-1}-w_2^{-1}}\star \bnab^2w_1+\p{w_2^{-1}-g^{-1}}\star\bnab^2\p{w_1-w_2}$$
and, in the $W^{(2),p}$ norm, it is bounded by
$$C\p{g,K}\p{\norm{w_1-w_2}_{C^{(2)}}\norm{w_1}_{W^{(4),p}}+\norm{w_2-g}_{C^{(2)}}\norm{w_1-w_2}_{W^{(4),p}}}\leq C(g,K)T^\gamma \norm{w_1-w_2}_{W^{(4),p}}.$$
Moving on to the first boundary operator $\Phi^1$, we will need the notation 
$$A_\delta(s,t)_m:=\underbrace{\p{h^\delta_{lm}g_{\delta,s}^{ik}-g^{\delta,s}_{lm}g_{\delta,s}^{ij}g_{\delta,s}^{nk}h^\delta_{jk}}}_{=:\p{H_\delta}_{lm}^{ik}}T\p{\tilde{g},g_{\delta,s}}_{ik}^l+\xi_t(g)_m,$$
$$B_\delta(s):=\beta_{g}\p{h_\delta}-\beta_{g_{\delta,s}}\p{h_\delta}$$
where $\delta=1,2$, while $h_\delta=w_\delta-g$ and $g_{\delta,s}=g+sh_\delta$. Recall that 
$$\Phi^1_{w_1}-\Phi^1_{w_2}=\int_{0}^1\p{B_1-B_2}-\p{A_1-A_2}\; ds.$$
We may write the $A$'s as
$$A_1-A_2=\p{\p{H_1}_{lm}^{ik}-\p{H_2}_{lm}^{ik}}T\p{\tilde{g},g_{1,s}}_{ik}^l+\p{H_2}_{lm}^{ik}T\p{g_{1,s},g_{2,s}}_{ik}^l.$$
Then, with yet another application of the Solonnikov embedding,
\begin{align*}\norm{A_1-A_2}_{B^{\p{3-\frac{1}{p}},p}}\leq & C\p{g,\tilde{g},K}\norm{H_1-H_2}_{C^{(2)}}\norm{\tilde{g}-g_{1,s}}_{W^{\p{4},p}}+C(g,K)\norm{H_2}_{C^{(2)}}\norm{g_{1,s}-g_{2,s}}_{W^{\p{4},p}}\\ 
\leq & C(g,\tilde{g},K)\norm{h_1-h_2}_{C^{(2)}}+C(g,K)\norm{h_2}_{C^{(2)}}\norm{w_1-w_2}_{W^{(4),p}}.\\
\leq & C(g,\tilde{g},K)T^\gamma \norm{w_1-w_2}_{W^{(4),p}} \end{align*}
As for the $B$'s, we have 
\begin{align*}B_1-B_2=\beta_g\p{h_1-h_2}-\beta_{g_{1,s}}\p{h_1-h_2}+\beta_{g_{2,s}}\p{h_2}-\beta_{g_{1,s}}\p{h_2}.
\end{align*}
The first two Bianchi operators can be written as
\begin{align*}\beta_g\p{h_1-h_2}-\beta_{g_{1,s}}\p{h_1-h_2} =&\p{g^{-1}-g_{1,s}^{-1}}\star \bnab\p{h_1-h_2}\\&+\p{\p{g^{-2}-g_{1,s}^{-2}}\star \bnab g+g_{1,s}^{-2}\star \bnab \p{g-g_{1,s}}}\star \p{h_1-h_2}\end{align*}
from which we get 
\begin{align*}\norm{\beta_g\p{h_1-h_2}-\beta_{g_{1,s}}\p{h_1-h_2}}_{B^{\p{3-\frac{1}{p}},p}} &\leq C\p{g,K}\p{\norm{h_1}_{C^{(2)}}\norm{h_1-h_2}_{W^{(4),p}}+\norm{h_1}_{W^{(4),p}}\norm{h_1-h_2}_{C^{(2)}} }\\ &\leq C\p{g,K}T^\gamma\norm{w_1-w_2}_{W^{(4),p}}.\end{align*}
Similarly,
\begin{align*}\norm{\beta_{g_{2,s}}\p{h_2}-\beta_{g_{1,s}}\p{h_2}}_{B^{\p{3-\frac{1}{p}},p}} &\leq C\p{g,K}\p{\norm{g_{1,s}-g_{2,s}}_{C^{(2)}}\norm{h_2}_{W^{(4),p}}+\norm{g_{1,s}-g_{2,s}}_{W^{(4),p}}\norm{h_2}_{C^{(2)}}}\\ &\leq C\p{g,K}T^\gamma \norm{w_1-w_2}_{W^{(4),p}}.\end{align*}
We can now conclude that 
$$\norm{\Phi_{w_1}^1-\Phi_{w_2}^1}_{B^{\p{3-\frac{1}{p}},p}}\leq\norm{A_1-A_2}_{B^{\p{3-\frac{1}{p}},p}}+\norm{B_1-B_2}_{B^{\p{3-\frac{1}{p}},p}}\leq C(g,\tilde{g},K)T^\gamma\norm{w_1-w_2}_{W^{(4),p}}. $$
Even now, the second operator is no obstacle:
$$\Phi^2_{w_1}-\Phi^2_{w_2}=\p{w_1^{-1}\star w_1-w_2^{-1}\star w_2}\star k_1+\p{w_2^{-1}\star w_2-g^{-1}\star g}\star \p{k_1-k_2}$$
where $k_1=\partial_t w_1$ and $k_2=\partial_t w_2$. By the Solonnikov embedding; $\norm{k_1-k_2}_{B^{\p{2-\frac{1}{p}},p}}\leq \norm{w_1-w_2}_{W^{(4),p}}$. Thus
\begin{align*}\norm{\Phi^2_{w_1}-\Phi^2_{w_2}}_{B^{\p{2-\frac{1}p},p}}&\leq C(g,K)\norm{w_1-w_2}_{C^{(2)}}\norm{w_1}_{W^{(4),p}}+C(g,K)\norm{w_2-g}_{C^{(2)}}\norm{w_1-w_2}_{W^{(4),p}}\\
&\leq C(g,K)T^\gamma\norm{w_1-w_2}_{W^{(4),p}}.	
\end{align*}
As for the third boundary operator, we have    
\begin{align*} \Phi^3_{w_1}-\Phi^3_{w_2}=&\MC'_g\p{k_1-k_2}-\MC'_{w_2}\p{k_1-k_2}+\p{\MC'_{w_2}\p{k_1}-\MC'_{w_1}\p{k_1}}\\
&+w_2^{-1}\star k_2\star \p{\MC_{w_2}-\MC_{w_1}}+\p{w_1^{-1}\star k_1-w_2^{-1}\star k_2}\star \MC_{w_1}.
\end{align*}
Recall the compact notation:
$$\MC_g'(k)=\p{g^{-2}+g^{-1}}\star \bnab \p{g\star \nu \star k}.$$
Thus, if $\dnu_2$ denotes the unit normal field of $w_2$ 
\begin{align*}\MC'_g\p{k_1-k_2}-\MC'_{w_2}\p{k_1-k_2}=&\p{g^{-2}-w_2^{-2}+g^{-1}-w_2^{-1}}\star \bnab \p{g\star \nu \star \p{k_1-k_2}}\\ & +\p{w_2^{-2}+w_2^{-1}}\star \bnab\p{\p{g\star \nu-w_2\star \dnu_2}\star \p{k_1-k_2}}
\end{align*}
and, by Solonnikov embedding and the usual tricks:
\begin{align*}\norm{\MC'_g\p{k_1-k_2}-\MC'_{w_2}\p{k_1-k_2}}_{B^{\p{1-\frac{1}{p}},p}}\leq & C(g,K)\norm{g-w_2}_{C^{(0)}}\norm{g\star \nu\star \p{k_1-k_2}}_{W^{(2),p}}\\
&+C(g,K)\norm{\p{g\star \nu-w_2\star \dnu_2}\star \p{k_1-k_2}}_{W^{(2),p}}\\ \leq &C(g,K)\norm{g-w_2}_{C^{(2)}}\norm{w_1-w_2}_{W^{(4),p}}\\  \leq & C(g,K)T^\gamma \norm{w_1-w_2}_{W^{(4),p}}.\end{align*}
Similarly,
\begin{align*}
	\MC'_{w_1}\p{k_1}-\MC'_{w_2}\p{k_1}=& \p{w_1^{-2}-w_2^{-2}+w_1^{-1}-w_2^{-1}}\star \bnab\p{w_1\star \dnu_1\star k_1}\\ &+\p{w_2^{-2}+w_2^{1}}\star \bnab \p{\p{w_1\star \dnu_1-w_2\star \dnu_2}\star k_1}
\end{align*}
which, just as before, suggests 
\begin{align*}
\norm{\MC'_{w_1}\p{k_1}-\MC'_{w_2}\p{k_1}}_{B^{\p{1-\frac{1}{p}},p}}\leq & C(g,K)\norm{w_1-w_2}_{C^{(0)}}\norm{w_1\star \dnu_1\star k_1}_{W^{(2),p}}\\ & +C(g,K)\norm{\p{w_1\star \dnu_1-w_2\star \dnu_2}\star k_1}_{W^{(2),p}}\\ \leq & C(g,K)\norm{w_1-w_2}_{C^{(2)}}\norm{w_1}_{W^{(4),p}}\\  \leq & C(g,K)T^\gamma \norm{w_1-w_2}_{W^{(4),p}}.
\end{align*}
As for the difference of mean curvatures,
\begin{align*}
\MC_{w_1}-\MC_{w_2}=w_1^{-1}\star \bnab \p{\dnu_1\star \p{w_1-w_2}+w_2\star \p{\dnu_1-\dnu_2}}+\p{w_1^{-1}-w_2^{-1}}\star \bnab\p{w_2\star \dnu_2}.
\end{align*}
Then 
\begin{align*}
\norm{\MC_{w_1}-\MC_{w_2}}_{B^{\p{1-\frac{1}{p}}}}\leq &C(g,K)\norm{\dnu_1\star \p{w_1-w_2}}_{W^{(2),p}}+C(g,K)\norm{w_2\star \p{\dnu_1-\dnu_2}}_{W^{(2),p}}\\ &+C(g,K)\norm{w_1-w_2}_{C^{(0)}}\norm{w_2\star \dnu_2}_{W^{(2),p}}	\\ \leq & C(g,K)\norm{w_1-w_2}_{W^{(2),p}}.
\end{align*}
We may now estimate the remainder of the third boundary operator:
\begin{align*}
&\norm{w_2^{-1}\star k_2\star \p{\MC_{w_1}-\MC_{w_2}}+\p{w_1^{-1}-w_2^{-1}}\star k_1\star \MC_{w_1}+w_2^{-1}\star \p{k_1-k_2}\star \MC_{w_1}}_{B^{\p{1-\frac{1}{p}},p}}\\ & \quad \leq C(g,K)\p{\norm{w_2}_{C^{(2)}}\norm{w_1-w_2}_{W^{(2),p}}+\norm{w_1-w_2}_{C^{(0)}}\norm{w_1}_{C^{(1)}}\norm{w_1}_{W^{(2),p}}+\norm{w_1}_{C^{(1)}}\norm{w_1-w_2}_{W^{(2),p}}} \\ &\quad \leq C(g,K)\norm{w_1-w_2}_{W^{(2),p}}\\ & \quad \leq C(g,K)T^{\frac{1}{p}}\norm{w_1-w_2}_{W^{(4),p}}
\end{align*}
where the final inequality follows from the Sobolev embedding. This was the last piece of the puzzle, and we have finally obtained the estimate 
\begin{align*}\norm{S(w_1)-S(w_2)}_{W^{(4),p}\p{M_T}} &\leq C(g)\norm{F_{w_1}-F_{w_2}}_{W^{(2),p}\p{M_T}}+C(g)\sum_{i=1}^3\norm{\Phi^i_{w_1}-\Phi^i_{w_2}}_{B^{\p{4-i-\frac{1}p},p}\p{\PS_T}}\\ &\leq C\p{g,\tilde{g},K}T^{\gamma_p} \norm{w_1-w_2}_{W^{(4),p}\p{M_T}},\end{align*}
which concludes the proof that we can choose $T$ so small that $S$ becomes a contraction of $M_K^T(g)$. The existence of a fixed point follows, and such a point is, per design, a solution of the Ricci-deTurck IBP (\ref{STERdT}).
\end{proof}

As the solution of a parabolic equation, the fixed point enjoys the the usual regularity properties in the interior $M^\circ_T$ (See e.g. \cite[Proposition 2.5]{Bre}). However, as it is now a boundary value problem, we are of course interested as to what degree of regularity we may expect at the boundary. We highlight this consideration with the following remark:

\begin{remark} Let $u_t$ be the $W^{(4),p}\p{M_T}$ solution of the Ricci-deTurck IBP, constructed in the preceding theorem. The boundary regularity of $u_t$ will a priori only be of class $C^{(3+\alpha)}$, for $\alpha\leq 1-\frac{n}{p}$. The next theorem improves the regularity, under the assumption that the initial metric $g$ satisfies higher order compatibility conditions.
\end{remark}	

\begin{theorem}[Boundary regularity of Ricci-deTurck flow] \label{BrRdt} Let $u\in W^{(4),p}(M_T)$ be a solution of the Ricci-deTurck initial boundary value problem (\ref{STERdT}). If $g$ satisfies compatibility conditions up to order $k\geq 1$ ($k=1$ is what we defined as being RdT-compatible) and $\tilde{g}\in C^{\p{k+3+\alpha}}(M_T)$, then $u\in C^{\p{k+2+\alpha}}(M_T)$.
\end{theorem}
\begin{proof} Assume $u\in C^{\p{m+\alpha}}\p{M_T}$, for some $3\leq m\leq k+1$. To obtain higher regularity, it is enough to show that $\hat{u}_l\in C^{(m+1+\alpha)}\p{\R^n_+\times [0,T]}$, where $\hat{u}_l$, as usual, is the local representation of $u$ on a boundary chart $U_l$. Since we shall only consider this chart, we ease the notational pressure by omitting the subscripted $l$, writing e.g. $\hat{\eta}$ for the partition of unity function with support in $V_l\subset U_l$ (see Definition \ref{fctspacesonM} for reference).

 Consider system (\ref{reallocalsystem}) with coefficient parameter $h=\hat{u}+\p{1-\hat{1}}\delta$ and input data
\begin{align*}
f_{ij}&=\hat{\eta}\mathcal{R}\p{\hat{u},\bnab\hat{u},\hat{\tilde{g}},\bnab\hat{\tilde{g}},\bnab^2\hat{\tilde{g}}}_{ij}-\hat{u}^{kq}\hat{u}_{ij}\bnab_k\partial_q \hat{\eta}-2\hat{u}^{kq}\bnab_q\hat{u}_{ij}\partial_k\hat{\eta}\\
\phi^1_m&= \p{n-\frac{1}{2}}\partial_m \hat{\eta}\\
\phi^2_{\alpha\beta}&=0\\
 \phi^3 &=-\frac{1}{2(n-1)}\hat{\eta} \hat{u}^{\alpha\beta}\hat{u}^{\gamma \epsilon}\nu^m	\partial_t \hat{u}_{\alpha\beta}\p{\partial_\gamma \hat{u}_{\epsilon m}-\frac{1}{2}\partial_m \hat{u}_{\gamma \epsilon}}-\hat{u}^{\alpha\beta}\nu^m\partial_t \hat{u}_{\beta m}\partial_\alpha \hat{\eta}+\frac{1}{2}\hat{u}^{\alpha\beta}\nu^m\partial_t \hat{u}_{\alpha\beta}\partial_m \hat{\eta}	
\end{align*}
where $\nu^m=-\p{h^{nn}}^{-1/2}h^{mn}$. One may readily verify that 
$$f_{ij}\in C^{\p{m-1+\alpha}}\p{\R^{n}_+\times [0,T]}\qquad \text{and}\qquad \phi^3\in C^{\p{m-2+\alpha}}\p{\partial\R^n_+\times [0,T]}.$$

An application of \cite[Theorem 4.9]{Sol1} implies $u\in C^{(m+1+\alpha)}\p{M_{T_0}} $ for some $0<T_0\leq T$. Solving the same system on $\R^n_+\times [T_0,T]$, with initial condition $\left.u\right|_{t=T_0}$ provides $u\in C^{\p{m+1+\alpha}}\p{M_{T_1}}$ for some $T_0<T_1\leq T$, and it may be repeated all the way up to time $T$. This procedure takes us from $u\in C^{(3+\alpha)}\p{M_{T}}$ to $u\in C^{(k+2+\alpha)}\p{M_{T}}$.
\end{proof}

\pagebreak

\section{Ricci Flow}

We are finally ready to apply the preceding theory to the Ricci flow. We begin by reminding the reader of the (harmonic) map Laplacian, with which we can find suitable diffeomorphisms to take us back and forth between the Ricci- and Ricci-deTurck flows. 

\begin{definition}[The map Laplacian] Let $(M,g)$ and $(N,h)$ be Riemannian manifolds and let $F\in C^\infty \p{M,N}$. If $\{x^{i}\}$ and $\{y^j\}$ are local coordinates on $M$ and $N$ respectively, we define the map Laplacian as
$$\p{\Delta_{g,h}F}^{i}=-\Delta_g F^{i}+g^{jk}\p{\Gamma(h)_{lm}^{i}\circ F}\frac{\partial F^l}{\partial x^j}\frac{\partial F^m}{\partial x^k}=g^{jk} \p{\frac{\partial^2 F^{i}}{\partial x^j\partial x^k}-\Gamma(g)_{jk}^{l}\frac{\partial F^{i}}{\partial x^l}+ \p{\Gamma(h)_{lm}^{i}\circ F}\frac{\partial F^l}{\partial x^j}\frac{\partial F^m}{\partial x^k}} $$
where $F^{i}:=y^{i}\circ F$.	If $M=N$ and $F=\Id$, then
	$$\Delta_{g,\tilde{g}_t}\Id=\xi_t^\sharp(g).$$
\end{definition}

The inherent problem of the Ricci flow - and the reason for first considering the Ricci-deTurck flow - is that it is not strictly parabolic. One way this manifests itself, is in the loss of regularity of solutions. To be a little more precise , if we were to apply the regularity theory of the preceding section to the Ricci flow, we would lose a derivative when compared with the Ricci-deTurck flow. Thus, to be able to apply the uniqueness of the Ricci-deTurck flow solution, we are forced to require that the initial metric satisfies compatibility conditions up to second order. This unfortunate requirement is enshrined in the following remark, where our old friends - the Einstein metrics - also make their final appearance.

\begin{definition} We will say that the initial metric $g_0$ is compatible, if under the Ricci flow it satisfies 
$$\tn{\partial_t^m\p{\partial_t \MC_{g_t}+\frac{1}{2(n-1)}\trace_{g_t^T}\p{\partial_t g_t}\MC_{g_t}}}=0\qquad \qquad m=0,1$$
$$\tn{\partial_t^m \p{\partial_t g_t^T-\frac{1}{n-1}\trace_{{g_t}^T}\p{\partial_t g_t}{g_t}^T}}=0\qquad \qquad m=0,1.$$
\end{definition}
\begin{remark}\label{compatiblemetric} 
For an Einstein initial metric $g$ with constant $\mu$, the first variation of mean curvature, Lemma \ref{MCfirstvar}, gives 
$$\tn{\partial_t\MC_{g_t}}=\mu \MC_g$$
showing the first condition is satisfied for $m=0$. The second is trivially satisfied for $m=0$ and $m=1$, for all Einstein metrics. Furthermore, 
$$\tn{\partial^2_t\MC_{g_t}}=6\mu\tn{\partial_t\MC_{g_t}}-3\mu^2\MC_g=3\mu^2\MC_g$$
by the $0$'th order version of the first condition. Similarly
$$\tn{\partial_t \trace_{g_t^T}\p{\partial_t g_t}\MC_{g_t}}=2\mu\trace_{g^T}\p{-2\Ric_g}\MC_g+\trace_{g^T}\p{-2\Ric_g}\tn{\partial_t\MC_{g_t}}=-6(n-1)\mu^2\MC_g$$
showing that Einstein metrics are indeed compatible.
\end{remark}

The final two theorems - which were presented in the introduction - mimic theorems \ref{STE} and \ref{BrRdt} for the Ricci-deTurck flow.

\begin{theorem}[Existence and uniqueness of Ricci flow] \label{existanduniq} Let $g_0$ be a compatible smooth Riemannian metric, then there exist a unique solution to the IBP
\begin{equation}\label{ricciflowibp}
\begin{cases}
\partial_t g_t=-2\Ric_{g_t} & \text{in } M_T\\
\begin{rcases}
	\MC'_{g_t}\p{\partial_t g_t}+\frac{1}2\varphi_{g_t}\p{\partial_t g_{t}}\MC_{g_t}=0\\
	\p{\partial_t g_t}^T-\varphi_{g_t}\p{\partial_t g_t}g_t^T=0
\end{rcases} &\text{on } \partial M_T\\
\left.g_{t}\right|_{t=0}=g_0
\end{cases}
\end{equation}	
such that $g_t\in W^{(4),p}\p{M_T}$.
\end{theorem}

\begin{proof} Choose a smooth family of background metrics $\tilde{g}_t$, with respect to which we define the deTurck Field. Let $u_t$ be the solution of (\ref{STERdT}) and, for some $t_0\geq 0$, solve the ODE
\begin{equation}\label{deturckode}\begin{cases}
\partial_t \psi_t(p)=\left.\xi_t^\sharp\p{u_t}\right|_{\psi_t(p)}, &  t\in  [0,T]\\
\psi_{t_0}(p)=p	
\end{cases}\end{equation}
 Consider the pullback $g_t:=\psi_t^*u_t$. Then
$$\partial_t g_t=\psi_t^*\p{\partial_t u_t}+\left.\pp{s}\right|_{s=0}\psi^*_{t+s}u_t=\psi_t^*\p{-2\Ric_{u_t}-\Lie_{\xi_t^\sharp(u_t)}u_t}+\Lie_{\p{\psi_t^{-1}}_*\xi_t^\sharp\p{u_t}}\psi_t^*u_t=-2\Ric_{g_t}.$$
Furthermore, the mean curvature is invariant under diffeomorphisms that fix the boundary, so $g_t$ satisfies the same boundary conditions as $u_t$. The regularity of $g_t$ follows from Theorem \ref{compandreg}.  \\

\noindent For uniqueness, let $g_{1,t},g_{2,t}\in W^{(4),p}\p{M_T}$ be solutions of the Ricci flow IBP. Define $\tau:=\inf\left\{t\in [0,T]\;;\; g_{1,t}\neq g_{2,t}\right\}$ and note that we must have $g_{1,\tau}=g_{2,\tau}=:g_\tau$. Solve the harmonic map heat flow
\begin{equation}\label{harmheat}\begin{cases}	\partial_t\psi_{i,t}(p)=\Delta_{g_{i,t},g_{\tau}}\psi_{i,t}(p)=\left.\xi_\tau^\sharp\p{\p{\psi_{i,t}}_* g_{i,t}}\right|_{\psi_{i,t}(p)}\\
\left.\psi_{i,t}\right|_{\partial M}=\Id_{\partial M}\\
\left.\psi_{i,t}\right|_{t=\tau}=\Id_{M}
\end{cases}
\end{equation}
By standard parabolic theory, this is possible on $M\times [\tau,\tau+\varepsilon)$ for some some $\varepsilon>0$. Clearly, $\ttau{\partial_t\psi_{i,t}}=\xi^\sharp_\tau\p{g_{\tau}}=0$, and since $\ttau{\partial_t g_{i,t}}=-2\Ric_{g_{\tau}}$ for both $i$,
$$\ttau{\partial_t^2 \psi_{i,t}}=\ttau{\partial_t \Delta_{g_{i,t},g_\tau}\psi_{i,t}}=-g_{\tau}^{jk}\ttau{\partial_t \Gamma\p{g_{i,t}}_{jk}^{i}}=-2\beta_{g_\tau}\p{\Ric_{g_\tau}}=0.$$
It follows that (\ref{harmheat}) satisfies compatibility conditions up to order $4$. This implies $\psi_{i,t}\in W^{(6),p}\p{M\times[\tau,\tau+\varepsilon)}$ and thus $u_{i,t}:=\p{\psi_{i,t}}_*g_{i,t}\in W^{(4),p}\p{M\times[\tau,\tau+\varepsilon)}$ are solutions of the Ricci-deTurck flow. By the uniqueness of such solutions, $u_{1,t}=u_{2,t}$ and $\psi_{1,t}$ and $\psi_{2,t}$ therefore satisfy the same ODE (\ref{harmheat}), implying $\psi_{1,t}=\psi_{2,t}$ and by extension $g_{1,t}=\p{\psi_{1,t}}^*u_{1,t}=\p{\psi_{2,t}}^*u_{2,t}=g_{2,t}$ for $t\in [\tau,\tau+\varepsilon)$. This contradiction proves equality all the way to $T$, by iterating the argument.
\end{proof}

\begin{theorem}[Compatibility conditions and regularity of Ricci flow]\label{compandreg} Let $g_t$ be a solution to the Ricci flow IBP (\ref{ricciflowibp}) for a smooth initial metric $g_0$ satisfying 
\begin{equation}\label{compcond1}\left.\partial_t^m\p{\partial_t \MC_g+\frac{1}{2(n-1)}\trace_{g^T}\p{\partial_t g}\MC_g}\right|_{t=0}=0\qquad \qquad m=0,\ldots,k-1\end{equation}
\begin{equation}\label{compcond2}\left.\partial_t^m\p{\partial_t g^T-\frac{1}{n-1}\trace_{g^T}\p{\partial_t g}g^T}\right|_{t=0}=0\qquad \qquad m=0,\ldots,k\end{equation}
Then $g_t\in C^{(2k+1+\alpha)}\p{M_T}$. If (\ref{compcond2}) only holds for $m\leq k-1$, then $g_t\in W^{(2k),p}\p{M_T}$.
\end{theorem}
\begin{proof} Choose a family of background metrics $\tilde{g}_t\in C^{(2k+3+\alpha)}\p{M_T}$ satisfying 
$$\tn{\tilde{g}_t}=g_0,\qquad \tn{\partial_t\tilde{g}_t}=0,\qquad \tn{\partial_t^m\tilde{g}_t}=\tn{\partial_t^m g_t} \quad \text{for }m=2,\ldots, k+1. $$ 
Define the deTurck Field with respect to this family of metrics and let $u_t$ be the solution of the corresponding Ricci-deTurck IBP (\ref{STERdT}) for the same initial metric. We now claim that $\tn{\partial_t^m u_t}=\tn{\partial_t^m g_t}$ for $m=0,\ldots,k+1$. It is certainly true for $m=0$, since $\tn{u_t}=g_0=\tn{g_t}$. Assuming it's also true for $m=1,\ldots, l-1$, for an $l\leq k+1$, we see that the difference in Christoffel symbols satisfy
\begin{equation}\label{partialdtvf}\tn{\partial_t^mT\p{\tilde{g}_t,u_t}}=0\qquad \text{for }m=2,\ldots,l-1\end{equation}
since $\tn{\partial^m_t u_t}=\tn{\partial_t^m \tilde{g}_t}$ for those $m$. By the contracted differential Bianchi identity, (\ref{partialdtvf}) also holds for $m=1$. It follows directly from the definition of the deTurck field, that
$$\tn{\partial_t^m \xi_t\p{u_t}}=0\qquad \text{for }m=0,\ldots,l-1 $$
from which follows, by induction:
$$\tn{\partial_t^l u_t}=\tn{\partial_t^{l-1}\p{-2\Ric_{u_t}-\Lie_{\xi_t\p{u_t}}u_t}}=-2\tn{\partial_t^{l-1}\Ric_{u_t}}=-2\tn{\partial_t^{l-1}\Ric_{g_t}}=\tn{\partial_t^{l}g_t}.$$
Thus $u_t$ satisfies 
$$\tn{\partial_t^m\xi_t(u_t)}=0,\qquad \qquad m=0,\ldots, k$$
$$\tn{\partial_t^m\p{\partial_t u_t^T-\frac{1}{n-1}\trace_{u_t^T}\p{\partial_t u_t}u_t^T}}=0,\qquad \qquad m=0,\ldots, k$$
$$\tn{\partial_t^m\p{\partial_t \MC_{u_t}+\frac{1}{2(n-1)}\trace_{u_t^T}\p{\partial_t u_t}\MC_{u_t}}}=0,\qquad \qquad m=0,\ldots,k-1$$
which are exactly compatibility conditions up to order $2k$, required for $u_t\in C^{(2k+2+\alpha)}\p{M_T}$. The deTurck field $\xi:=\xi_t(u_t)$ satisfies the variational formula
$$\partial_t\xi_i=\beta_u\p{h}_i+\p{h_{ij}u^{kl}-u_{ij}u^{km}u^{lp}h_{mp}}T\p{\tilde{g},u}_{kl}^j+\frac{1}{2}u_{ij}u^{kl}\partial_t\tilde{\Gamma}_{kl}^j=\beta_u\p{h}_i+P\p{u,\nabla u,\nabla^2 u,\tilde{g},\tilde{\nabla}\tilde{g},\partial_t\tilde{\nabla}\tilde{g}}_i$$
where $h=\partial_t u=-2\Ric_u-\Lie_{\xi^\sharp}u=-2\Ric_u-2\delta^*_u\xi$, $\nabla$ is the connection of $u$ and $P$ is some linear combination of contractions of the given entries. Applying the contracted differential Bianchi identity
$$\partial_t \xi=-2\beta_u\p{\Ric_u+\delta^*_u\xi}+P=-2\beta_u\p{\delta^*_u\xi}+P.$$
The Ricci formulae then gives 
$$-2\beta_u\p{\delta^*_u\xi}_i=-2\p{\delta_u\delta^*_u\xi}_i+\nabla_i\delta_u\xi=\nabla^{j}\p{\nabla_j\xi_i+\nabla_i\xi_j}-\nabla_i\nabla^j\xi_j=\nabla^j\nabla_j\xi_i+\Ric_{ij}\xi^j.$$
Thus, we have shown that $\xi$ satisfies the parabolic IBP
\begin{equation}\label{deturckipb}
\begin{cases}\partial_t\xi +\Delta_u \xi=\Ric_u\cdot \xi^\sharp+P & \text{in } M_T\\
\xi=0 & \text{on }\partial M_T\\ 
\tn{\xi}=0
\end{cases}
\end{equation}
Since $\tn{\partial_t^m u}=\tn{\partial_t^m \tilde{g}}$ for $m=0,\ldots, k+1$, it follows that the system (\ref{deturckipb}) satisfies compatibility conditions up to order $2k$. Since $P\in C^{\p{2k+\alpha}}\p{M_T}$, we can use \cite[Theorem 4.9]{Sol1} to obtain $\xi\in C^{\p{2k+2+\alpha}}\p{M_T}$.\\

\noindent The solution $\psi_t$ of (\ref{deturckode}) will now be in $C^{\p{2k+2+\alpha}}\p{M_T}$, which means $g_t=\psi_t^*u_t\in C^{\p{2k+1+\alpha}}\p{M_T}$. If (\ref{compcond2}) holds up to $k-1$, we have $u_t \in W^{(2k+2),p}\p{M_T}$ (since this only requires compatibility up to order $ 2k-1$) and thus $g_t\in W^{(2k),p}\p{M_T}$, by the same argument.
\end{proof}

\pagebreak

\end{document}